\documentclass[11pt]{amsart}
\usepackage{graphicx}
\usepackage{subfig}
\usepackage{amsmath, amsthm, amssymb}
\usepackage{mathrsfs}
\usepackage{algorithm}
\usepackage{algorithmic}
\usepackage[all]{xy}
\usepackage{graphicx}
\usepackage{caption}
\usepackage{mathptmx}
\usepackage[margin=2.5cm]{geometry}
\usepackage{url}
\usepackage{combelow}
\usepackage{subfig}
\usepackage{xcolor}
\usepackage[colorinlistoftodos,prependcaption,textsize=tiny]{todonotes}

\allowdisplaybreaks

\floatname{algorithm}{Algorithm}
\newcommand{\algref}[1]{Algorithm~\ref{#1}}

\theoremstyle{plain}
\newtheorem{thm}{Theorem}[section]
\newtheorem{prop}[thm]{Proposition}
\newtheorem{lemma}[thm]{Lemma}

\newtheorem{coro}[thm]{Corollary}
\theoremstyle{definition}
\newtheorem{defn}[thm]{Definition}
\theoremstyle{definition}

\newtheorem{ex}[thm]{Example}
\theoremstyle{remark}
\newtheorem{rem}[thm]{Remark}
\theoremstyle{definition}
\newtheorem*{thm*}{Theorem}
\newtheorem*{ex*}{Example}

\newcommand{\defref}[1]{Definition~\ref{#1}}
\newcommand{\propref}[1]{Proposition~\ref{#1}}
\newcommand{\thmref}[1]{Theorem~\ref{#1}}
\newcommand{\lemmaref}[1]{Lemma~\ref{#1}}
\newcommand{\cororef}[1]{Corollary~\ref{#1}}
\newcommand{\exref}[1]{Example~\ref{#1}}
\newcommand{\secref}[1]{Section~\ref{#1}}

\newcommand{\remref}[1]{Remark~\ref{#1}}

\newcommand{\figref}[1]{Figure~\ref{#1}}

\newcommand{\NN}{\mathbb{N}}
\newcommand{\ZZ}{\mathbb{Z}}
\newcommand{\RR}{\mathbb{R}}
\newcommand{\CC}{\mathbb{C}}

\newcommand{\PP}{\mathbb{P}}

\newcommand{\cOO}{\xi}

\newcommand{\im}[1]{\operatorname{im} (#1)}
\newcommand{\codim}[1]{\operatorname{codim} (#1)}

\newcommand{\Bl}[2]{\operatorname{Bl}_{#1}(#2)}
\newcommand{\gr}[2]{\operatorname{Gr} (#1,#2)}

\newcommand{\BND}[1]{\operatorname{BND}(#1)}

\renewcommand{\dim}[1]{\operatorname{dim} (#1)}
\renewcommand{\deg}[1]{\operatorname{deg} (#1)}

\begin{document}

\title{The bottleneck degree of algebraic varieties}

\author{Sandra Di Rocco}
\address{Department of mathematics, KTH, 10044,
  Stockholm, Sweden}
\email{dirocco@math.kth.se}
\urladdr{https://people.kth.se/~dirocco/}

\author{David Eklund} \address{DTU Compute, Richard Petersens Plads,
  Building 321, DK-2800 Kgs. Lyngby, Denmark}
\email{daek@math.kth.se}
\urladdr{https://people.kth.se/~daek/}

\author{Madeleine Weinstein} \address{Department of Mathematics,
  University of California Berkeley, Evans Hall, Berkeley, CA 94720}
\email{maddie@math.berkeley.edu}
\urladdr{https://math.berkeley.edu/~maddie/}


\begin{abstract}
A bottleneck of a smooth algebraic variety $X \subset \CC^n$ is a pair
$(x,y)$ of distinct points $x,y \in X$ such that the Euclidean normal
spaces at $x$ and $y$ contain the line spanned by $x$ and $y$. The
narrowness of bottlenecks is a fundamental complexity measure in the
algebraic geometry of data. In this paper we study the number of
bottlenecks of affine and projective varieties, which we call the
\emph{bottleneck degree}. The bottleneck degree is a measure of the
complexity of computing all bottlenecks of an algebraic variety, using
for example numerical homotopy methods. We show that the bottleneck
degree is a function of classical invariants such as Chern classes and
polar classes. We give the formula explicitly in low dimension and
provide an algorithm to compute it in the general case.
\end{abstract}

\maketitle

\section{Introduction} \label{sec:intro}

In this paper we study geometric properties of algebraic varieties
with applications to computational data science. Let $f_1,\dots,f_k
\in \RR[x_1,\dots,x_n]$ be polynomials. The associated algebraic
variety is the zero-set $X \subset \RR^n$ given by $X=\{x \in \RR^n:
f_1(x)=\dots=f_k(x)=0\}$. Polynomial systems of equations arise in
applications to natural science, engineering, computer science and
beyond. Examples include kinematics \cite{wampler_sommese_2011},
economics \cite{MCKELVEY1997411}, chemistry \cite{Minimair04},
computer vision \cite{KUKELOVA2010}, machine learning
\cite{Mehta2018TheLS} and optimization \cite{lasserre_2015}.
Polynomial systems can be analyzed naturally through the machinery of
algebraic geometry.  In the present study we concentrate on computing
and counting so-called \emph{bottlenecks} of an algebraic variety $X
\subset \RR^n$. This is the study of lines in $\RR^n$ orthogonal to
$X$ at two or more points. Such lines contribute to the computation of
the {\it reach}, see \secref{sec:reach}, and may be found by solving a
polynomial system (\ref{eq:realbn}). To be able to use the appropriate
tools from algebraic geometry we often have to move from the real
numbers to the algebraically closed field of complex numbers $\CC$, as
we illustrate below.  We will see that classical invariants such as
polar classes appear naturally and turn out to be essential to
obtaining a closed formula for the number of bottlenecks.  In our
opinion this work provides yet one more illustration that classical
algebraic geometry and in particular intersection theory are useful
and often necessary in applications such as data science.

\subsection{Bottlenecks and optimization}
Finding  lines orthogonal at two or more points is an optimization problem with algebraic
constraints. The focus of this paper is to determine, or bound, the
number of solutions to this optimization problem.

\begin{ex} \label{ex:ellipseintro}
Consider the ellipse $C \subset \RR^2$ defined by $f=x^2+y^2/2-1=0$. A
bottleneck on $C$ is a pair of points $p,q \in C$ that span a line
orthogonal to $C$ at both points. The only such lines are the $x$-axis
and the $y$-axis, that is the principal axes of the ellipse, see
\figref{fig:ellipseintro}.
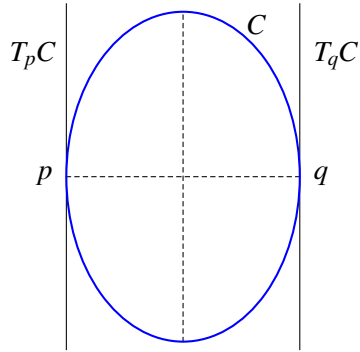
\begin{figure}[ht]
  \centering
\begin{picture}(150,150)
  \put(0,0){\includegraphics[trim={0pt 0pt 0pt 0pt},clip,scale=0.4]{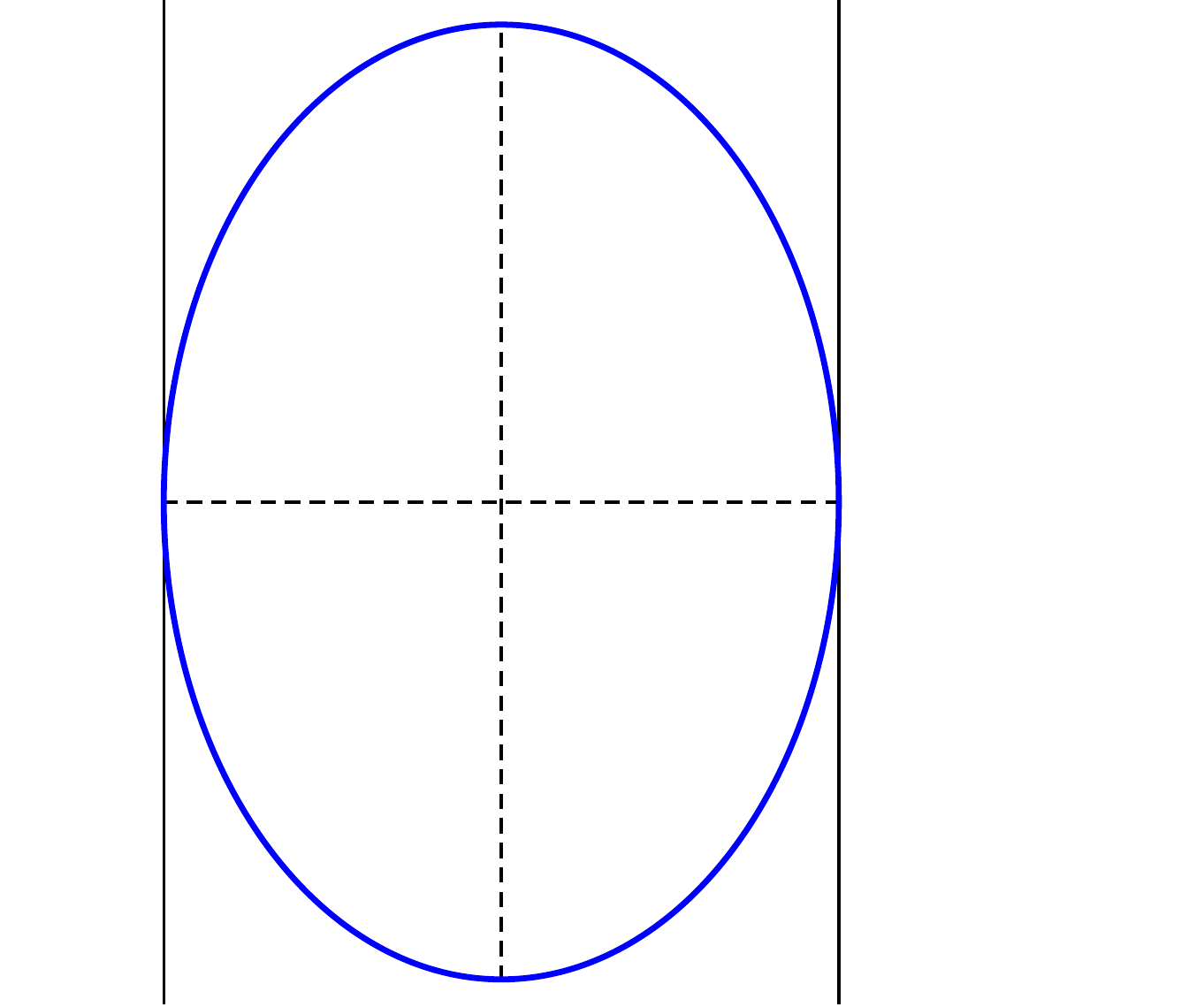}}
      \put(0,110){$T_pC$}
      \put(90,120){$C$}
      \put(115,110){$T_qC$}
      \put(10,65){$p$}
      \put(115,65){$q$}
    \end{picture}
\caption{An ellipse with tangent lines and principal axes.\label{fig:ellipseintro}}
\end{figure}
A line $l$ is orthogonal to $C$ at a point $p \in C$
if $l$ is orthogonal to the tangent line $T_pC$ at $p$. In
other words $l$ is the normal line $N_pX$ at $p$. The direction of the
normal line is given by the gradient $\nabla f=(2x,y)$. Consider a
pair of points $p=(x,y) \in C$ and $q=(z,w) \in C$.  The claim that
$(p,q)$ is a bottleneck may then be expressed as
\[
\begin{array}{lll}
x-z &=& \lambda 2x, \\
y-w &=& \lambda y, \\
x-z &=& \mu 2z, \\
y-w &=& \mu w,
\end{array}
\]
for some $\lambda, \mu \in \RR$. These equations, together with
$x^2+y^2/2=1$ and $z^2+w^2/2=1$, constitute a polynomial system for
computing bottlenecks on the curve $C$. Note that this is also the
system we get if we apply the Lagrange multiplier method to the
problem of optimizing the squared distance function $(x-z)^2+(y-w)^2$
subject to the constraints $x^2+y^2/2-1=z^2+w^2/2-1=0$. This is thus
an optimization problem and we are asking for the critical points of
the distance between pairs of points on $C$.
\end{ex}

Consider again an arbitrary variety $X \subset \RR^n$. For
convenience, we will restrict to the case where $X$ is \emph{smooth},
that is every point of $X$ is a manifold point. A line is orthogonal
to $X$ if it is orthogonal to the tangent space $T_xX \subset \RR^n$
at $x$.
\begin{defn}
  Let $X \subset \RR^n$ be a smooth variety. The bottlenecks of $X$
  are pairs $(x,y)$ of distinct points $x,y \in X$ such that the line
  spanned by $x$ and $y$ is normal to $X$ at both points.
\end{defn}
Equivalently  one can define bottlenecks as  the
critical points of the squared distance function
\begin{equation} \label{eq:sqdist}
  \RR^n \times \RR^n: (x,y) \mapsto ||x-y||^2,
  \end{equation}
subject to the constraints $x,y \in X$ as well as the non-triviality
condition $x\neq y$.

\begin{figure}[ht]
  \centering
\subfloat[][\label{fig:quarticcurve}]{
  \includegraphics[trim={10pt 10pt 10pt 10pt},clip,scale=0.15]{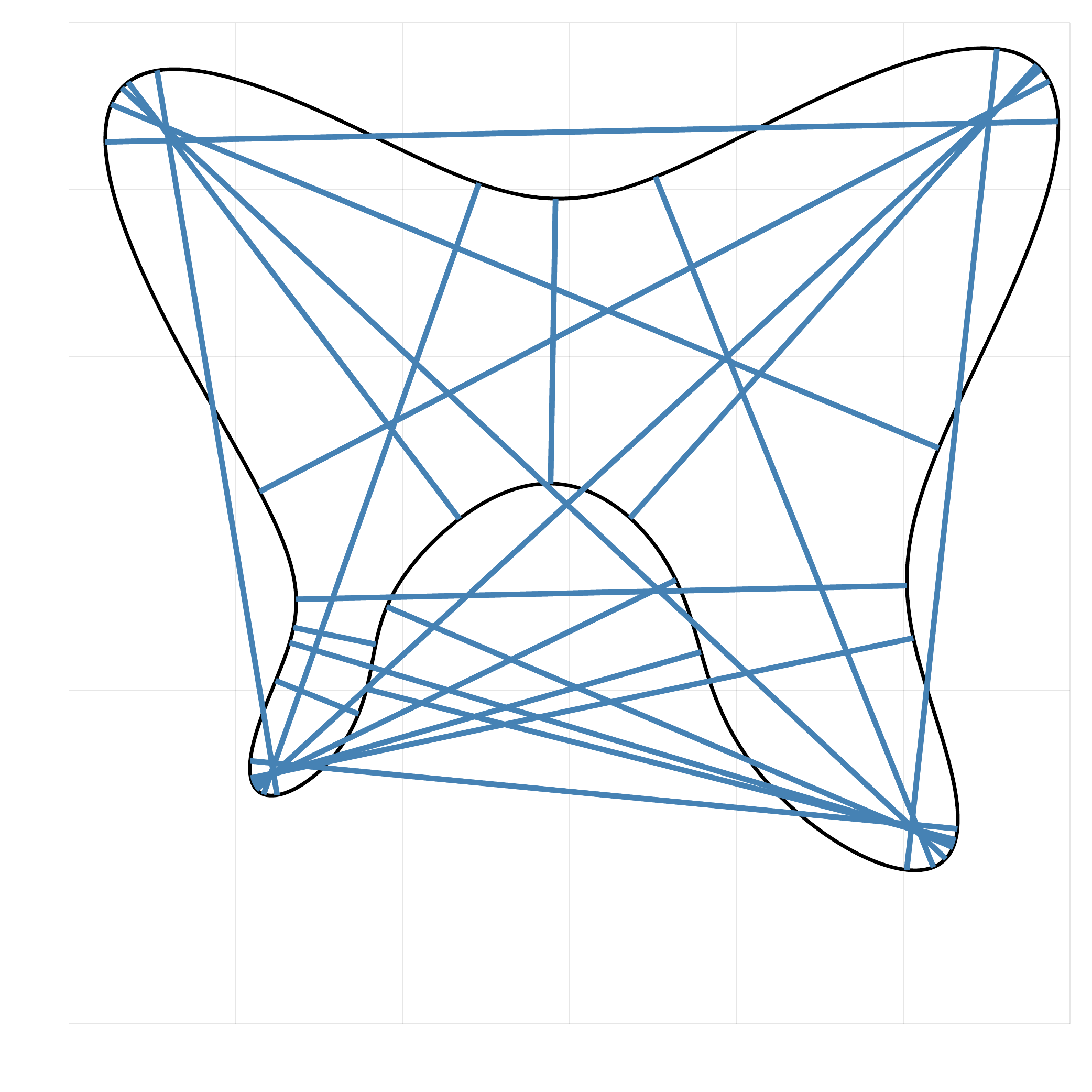}}
\qquad
\qquad
\subfloat[][\label{fig:realbncurve}]{
  \includegraphics[trim={10pt 10pt 10pt 10pt},clip,scale=0.3]{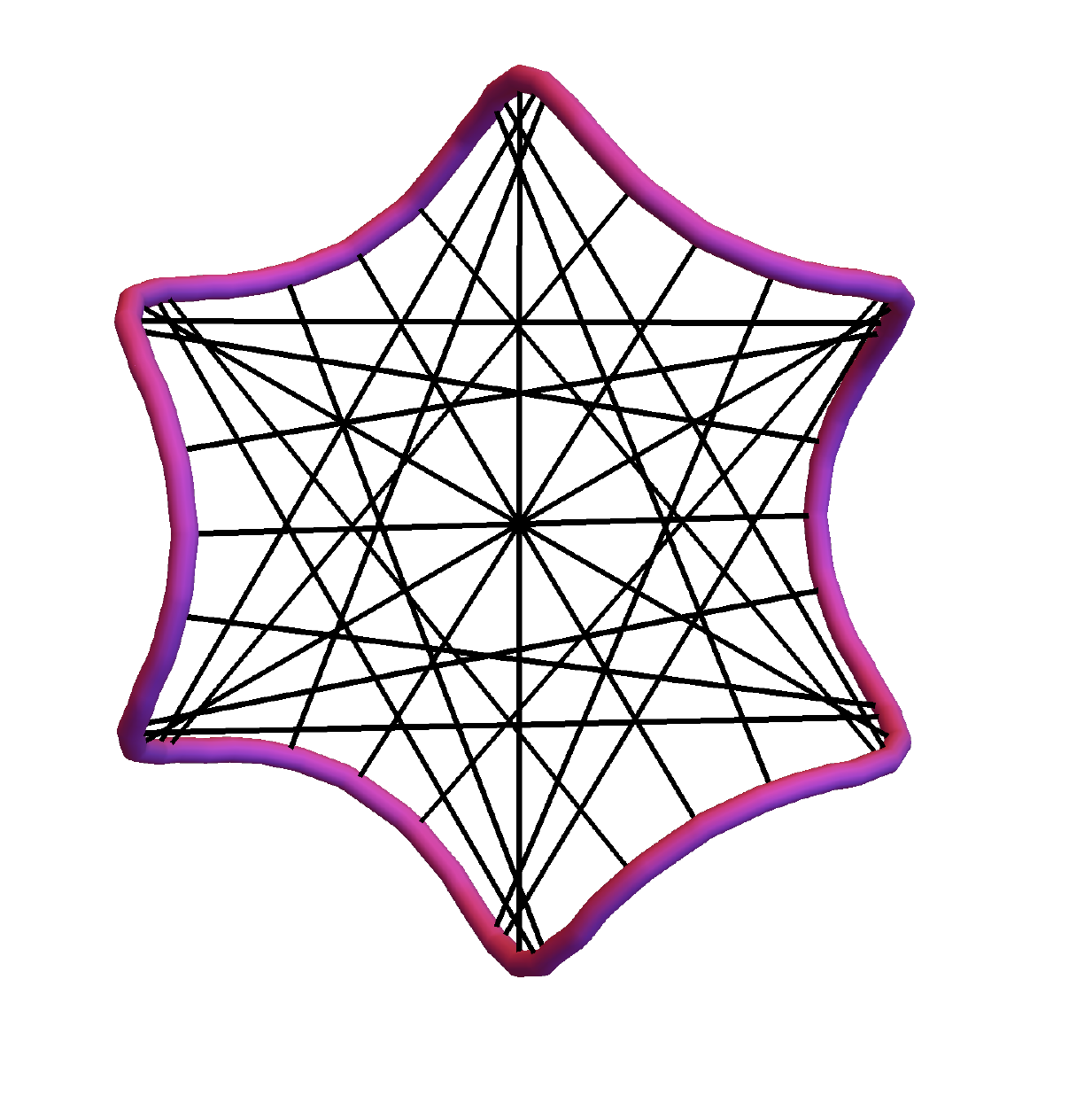}}
\caption{Two curves and their bottlenecks.}
\end{figure}

\begin{ex}
\figref{fig:quarticcurve} shows a quartic curve in $\RR^2$ and its 22
bottleneck lines. The curve is defined by $x^4 + y^4+1-4y -
x^2y^2-4x^2-x-2y^2=0$. The figure was produced by Paul Breiding and
Sascha Timme using the Julia package \emph{HomotopyContinuation.jl} \cite{BreidingTimme}.

As another example consider the space curve in $\RR^3$ defined by
   \[
 \begin{array}{c}
   x^3-3xy^2-z = 0,\\
   x^2 + y^2 + 3z^2 - 1 = 0.
 \end{array}
 \]
\figref{fig:realbncurve} shows this curve and its 24 bottleneck lines.

\end{ex}

\subsection{Motivation}\label{sec:reach}

The geometry of bottlenecks plays an important role in several aspects
of real geometry in connection with analysis of data with support on
an algebraic variety. Let $X_{\RR} \subset \RR^n$ be a smooth variety
which is non-empty and compact. An interesting observation is that the
distance between any two distinct connected components of $X_{\RR}$ is
realized by $||x-y||$ for some bottleneck $(x,y)$ with one point on
each component. See \figref{fig:components} for an illustration of
this. The narrowest bottleneck thus bounds the smallest distance
between any two connected components of $X_{\RR}$. This relates
bottlenecks and the so-called \emph{reach} $\tau_{\RR}$ of
$X_{\RR}$. The number $\tau_{X_{\RR}}$ can be defined as the maximal
distance $r \geq 0$ such that any point $p \in \RR^n$ at distance less
than $r$ from $X_{\RR}$ has a unique closest point on $X_{\RR}$.

The reach can be seen as a measure of curvature that extends to
subsets of $\RR^n$ which are not smooth manifolds; see \cite{F59} for
background and basic facts. The reach has many applications in the
area of manifold reconstruction \cite{ACKMRW17}. For example, it has
been applied to minimax rates of convergence in Hausdorff distance
\cite{Genovese2012, Kim2015TightMR, Aamari2018}. The reach has also
been applied to estimate boundary curve length and surface area
\cite{Cuevas2007ANA} as well as volume \cite{minimaxvolume}. Another
application is dimensionality reduction via random projections
\cite{B06}. In a number of papers \cite{BW09, C08, BHW08, V11} the
reach is used to bound the approximation error of such dimensionality
reduction techniques. This amounts to a generalization of the
Johnson-Lindenstrauss lemma \cite{DG03} to higher dimensional
manifolds. The reach is also used as input to standard algorithms for
manifold triangulation \cite{ACM05, BG11,BO05,Boissonnat2014}. Another
line of work seeks to compute {\it homology groups} of a manifold from
a finite sample using techniques from persistent homology
\cite{BKSW18, DEHH18, HW18, NSW08}. A point cloud on a torus is
illustrated in \figref{fig:components}. In this context, see also
\cite{balakrishnan2012} on minimax rates of homology inference. The
reach determines the sample size required to obtain the correct
homology of the associated complex. With these applications in mind it
would be useful to find efficient methods to compute the reach.

It is shown in \cite{ACKMRW17} that $\tau_{X_\RR}=\min \{\rho,b\}$
where $\rho$ is the minimal radius of curvature of a geodesic on
$X_{\RR}$ and $b$ is half the width of the narrowest bottleneck of
$X_{\RR}$. This suggests that $\tau_{X_{\RR}}$ could be computed with
numerical methods by computing $\rho$ and $b$ separately. See
\cite{reach-curve2019}, where this is carried out in detail for plane
curves. We conclude that efficient methods to compute bottlenecks play
a prominent role in the pursuit of efficient methods to compute the
reach itself.

\begin{figure}
    \centering
    \subfloat[Connected components]{{\includegraphics[trim={0pt 50pt 0pt 100pt},clip,width=150pt]{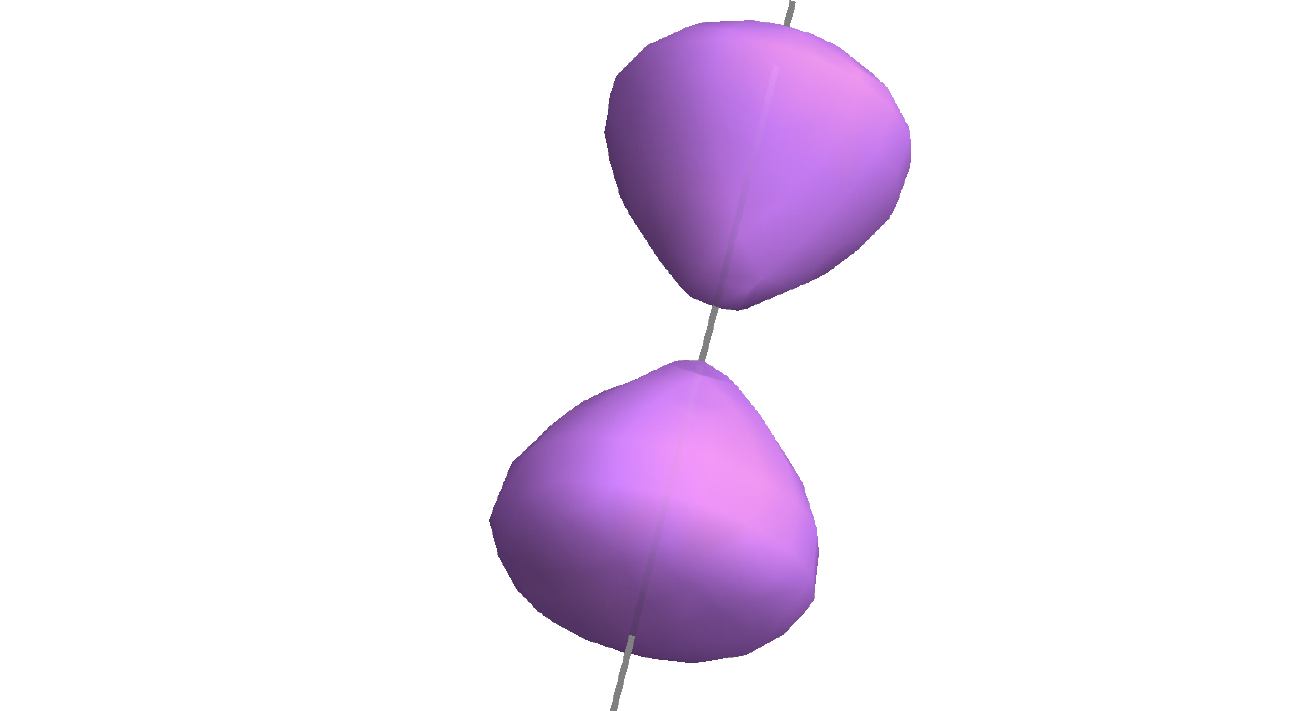} }}
    \qquad
    \subfloat[Sampling of a torus]{{\includegraphics[trim={0pt 50pt 0pt 100pt},clip,width=170pt]{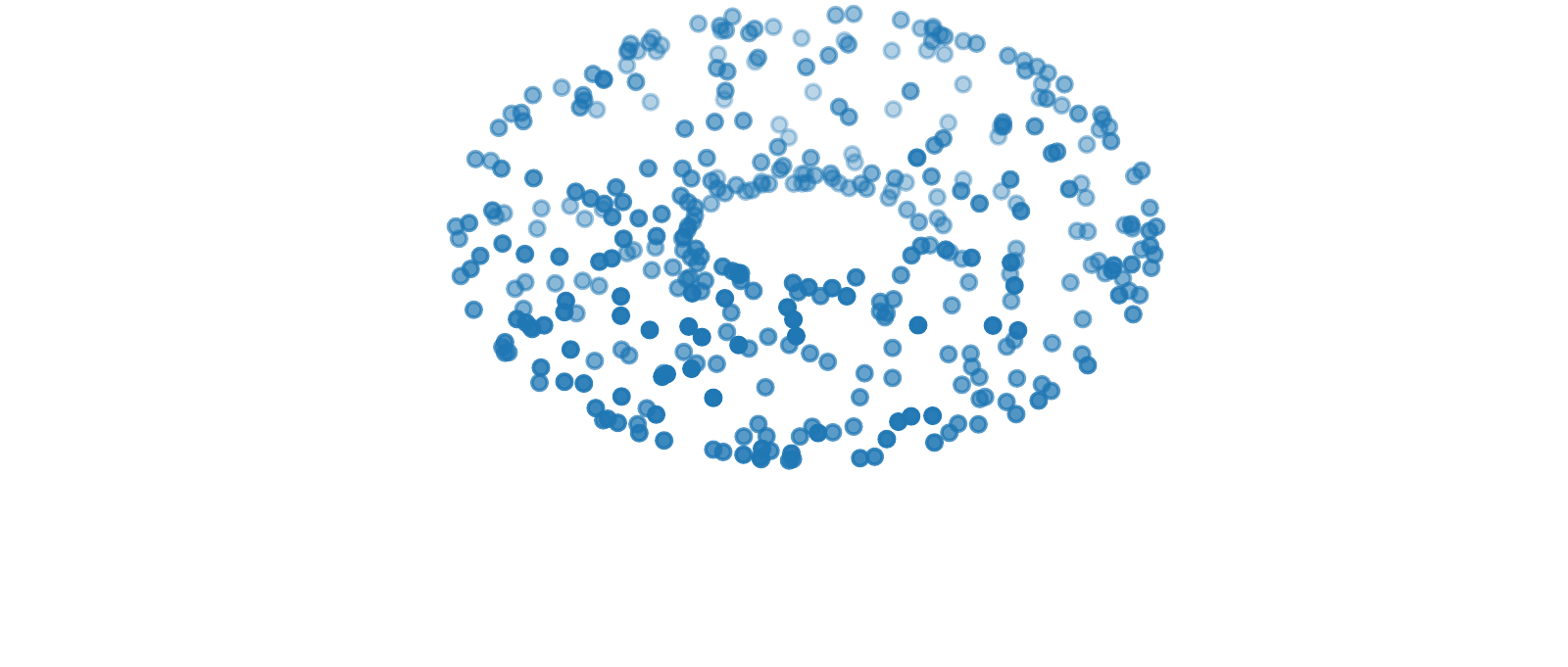} }}
    \caption{  }
    \label{fig:components}
\end{figure}

\subsection{Equations for bottlenecks}
We will now formulate a system of equations for bottlenecks that does
not introduce auxiliary variables as in the Lagrange multiplier
method. Both of these formulations are useful and the latter will be
developed further in \remref{rem:homotopy}.

Let $X \subset \RR^n$ be a smooth $m$-dimensional variety defined by
polynomials $f_1,\dots,f_k$. Note that for $x \in X$,
$\dim{T_xX}=\dim{X}=m$. Here we are considering the embedded tangent
space which passes through the point $x$. The corresponding linear
space through the origin is $(T_xX)_0=T_xX-x$. The orthogonal
complement $N_xX=\{z \in \RR^n: (z-x) \perp (T_xX)_0\}$ is the normal
space at $x$ and has the complementary dimension $n-m$. As in the case
of the ellipse in \exref{ex:ellipseintro}, the normal space is the
span of the gradients $\langle \nabla f_1,\dots, \nabla f_k
\rangle$. More precisely $N_xX=x+\langle \nabla f_1(x),\dots, \nabla
f_k(x) \rangle$. Now, if $x,y \in X$ are distinct then $(x,y)$ is a
bottleneck precisely when $(y-x) \in \langle \nabla f_1(x),\dots,
\nabla f_k(x) \rangle$ and $(y-x) \in \langle \nabla f_1(y),\dots,
\nabla f_k(y) \rangle$. To formulate the equations we define the
\emph{augmented Jacobian} to be the following matrix of size $(k+1)
\times n$:
\begin{equation} \label{eq:augmented}
 J(x,y) =
 \begin{bmatrix}
     y-x \\
     \nabla f_1(x) \\
     \vdots \\
     \nabla f_k(x)
 \end{bmatrix} ,
 \end{equation}
where $y-x$ is viewed as a row vector. The condition that $y-x$ is in
the span of $\nabla f_1(x),\dots,\nabla f_k(x)$ is equivalent to
saying that the matrix $J(x,y)$ has rank less than or equal to $n-m$,
or in other words that all $(n-m+1)\times (n-m+1)$-minors of $J(x,y)$
vanish. There is a similar rank condition given by the $(n-m+1) \times
(n-m+1)$-minors of the augmented Jacobian $J(y,x)$ with $x$ and $y$
reversed. In summary, the bottlenecks of $X$ are the non-trivial
($x\neq y$) solutions to the following system of equations:
\begin{equation} \label{eq:realbn}
  \begin{array}{l}
(n-m+1)\times (n-m+1)\text{-minors of }J(x,y)=0,\\
(n-m+1)\times (n-m+1)\text{-minors of }J(y,x)=0,\\
f_1(x)=\dots =f_k(x)=0,\\
f_1(y)=\dots=f_k(y)=0.
\end{array}
\end{equation}

\subsection{Counting roots, complex numbers and projective space} \label{sec:counting}

In this section we will motivate the study of complex and projective
bottlenecks. Let $f_1,\dots,f_k \in \RR[x_1,\dots,x_n]$ with
corresponding variety $X_{\RR}$. The system of equations
$f_1(x)=\dots=f_k(x)=0$ may have non-real solutions $x \in \CC^n$. The
complex solutions are very relevant for solving polynomial systems. We
can define a complex variety $X_{\CC} \subset \CC^n$ given by
$X_{\CC}=\{x \in \CC^n: f_1(x)=\dots=f_k(x)=0\}$. Note that $X_{\CC}$
contains the real solutions $X_{\RR} \subseteq X_{\CC}$.

As practitioners we need tools to numerically approximate solutions to
polynomial systems. A useful approach we would like to mention here is
numerical homotopy methods, see for example
\cite{SW05,bertini,BreidingTimme}. These are predictor/corrector
routines based on Newton's method but with probabilistic guarantees
that all complex isolated solutions will be found. If a system has
only finitely many solutions then the number of complex roots is an
upper bound on the number of real roots. A naive approach to finding
the real roots is of course to compute all complex roots and filter
out the real ones. We stress this point because it illustrates how the
number of complex bottlenecks (if finite) provides upper bounds on the
computational complexity of real bottlenecks. It is therefore natural
to explore the concept of bottlenecks in the complex setting even if
one is only interested in real solutions.

An alternative approach to homotopy methods is symbolic
computations  via Gr\"obner bases, see for example
\cite{Sturmfels02}. Whether homotopy methods or Gr\"obner bases is
appropriate depends on the particular system of equations at hand. See
\cite{Bates2014} for a comparison of numerical and symbolic methods
for equation solving.

Let $X \subset \CC^n$ be a smooth variety, defined by $f_1,\dots, f_k
\in \CC[x_1,\dots x_n],$ with $\dim{X}=m>0$. A \emph{bottleneck} of
$X$ is defined to be a pair of distinct points $x,y \in X$ such that
the line $\overline{xy}$ joining $x$ and $y$ is normal to $X$ at both
$x$ and $y$. The orthogonality relation $a\perp b$ involved in the
definition of bottlenecks is given by $\sum_{i=1}^na_ib_i=0$ for
$a=(a_1,\dots,a_n) \in \CC^n$ and $b=(b_1,\dots,b_n) \in \CC^n$. For a
point $x \in X$, let $(T_xX)_0$ denote the embedded tangent space of
$X$ translated to the origin. Then the \emph{Euclidean normal space}
of $X$ at $x$ is defined as $N_xX=\{z \in \CC^n:(z-x) \perp
(T_xX)_0\}$. A pair of distinct points $(x,y) \in X \times X$ is thus
a bottleneck exactly when $\overline{xy} \subseteq N_xX \cap
N_yX$. Note that this is the case if and only if $y \in N_xX$ and $x
\in N_yX$. The \emph{bottleneck variety} in $\CC^{2n}$ consists of the
bottlenecks of $X$ together with the diagonal $\{(x,y) \in X\times
X:x=y\} \subset \CC^n \times \CC^n$. Just as for real varieties, the
augmented Jacobian is defined by (\ref{eq:augmented}) and the system
(\ref{eq:realbn}) defines the bottleneck variety of $X$.

In a similar manner we will define bottlenecks for projective
varieties in complex projective space $\PP^n$. Recall that projective
space $\PP^n$ is obtained by gluing a hyperplane at infinity to the
affine space $\CC^n$. For example, the projective plane $\PP^2$ is the
complex plane $\CC^2$ with an added line at infinity.

Counting the number of roots to a system of polynomials is a highly
challenging problem. The simplest case is counting roots in
$\PP^n$. Counting roots in $\CC^n$ is harder and even harder is to
count real roots. Consider for example a univariate polynomial $f \in
\RR[x]$ of degree $d$. In this case there are always $d$ complex roots
counted with multiplicity while the number of real roots depends on
the coefficients of $f$. Consider now the next step of two equations
$f_1,f_2 \in \RR[x_1,x_2]$ of degrees $d_1$ and $d_2$ and the
corresponding intersection of two curves in $\CC^2$. If the
intersection is finite there can be at most $d_1d_2$ complex
solutions. This is also the number of roots for almost all $f_1$ and
$f_2$ of degrees $d_1$ and $d_2$. In the case of real curves in
$\RR^2$ there is no such generic root count. Also, the number of
complex intersection points may be smaller than $d_1d_2$ as
illustrated by the example of two disjoint lines defined by $x_1=0$
and $x_1=1$. In contrast, the intersection of two curves in $\PP^2$ of
degrees $d_1$ and $d_2$ with finite intersection always consists of
$d_1d_2$ points counted with multiplicity. This fact is known as
B\'ezout's theorem and it can be generalized to a system of $n$
equations in $n$ variables \cite[Proposition 8.4]{F98}. This might
seem to solve the problem, at least in $\PP^n$. However, the system we
want to solve might be \emph{overdetermined} and have \emph{excess
  components} of higher dimension. Both of these complications are
present in the system \eqref{eq:realbn}. The excess component in this
case consists of the discarded trivial solutions on the diagonal
$\{(x,y) \in \CC^n \times \CC^n: x=y\}$. Amazingly, intersection
theory provides tools to deal with these issues under certain
circumstances. These tools are however often confined to the complex
projective setting.

Bottlenecks for projective varieties turn out to be essential for
counting bottlenecks on affine varieties. In fact, in
\propref{prop:affine} we reduce the affine case to the projective case
by considering bottlenecks at infinity.

\subsection{Polar geometry}

Consider the ellipse $C$ defined by $x^2+y^2/2=1$ in
\exref{ex:ellipseintro}. For a point $p \in \RR^2$ outside the region
bounded by $C$ there are exactly two lines through $p$ tangent to
$C$. The two tangent points $x,y \in C$ define what we call the first
polar locus $P_1(X,p)=\{x,y\}$, see \figref{fig:curvepolar}. The polar
locus depends on the choice of $p$ but two polar loci $P_1(X,p)$ and
$P_1(X,p')$ can be seen as deformations of each other by letting $p'$
approach $p$ along a curve. In this sense, the polar loci all
represent the same \emph{polar class} $p_1$ on $C$.

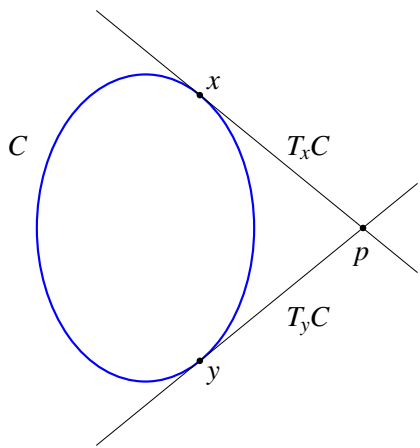
\begin{figure}[ht]
  \centering
  \begin{picture}(160,160)
  \put(0,0){\includegraphics[trim={0pt 0pt 0pt 0pt},clip,scale=0.3]{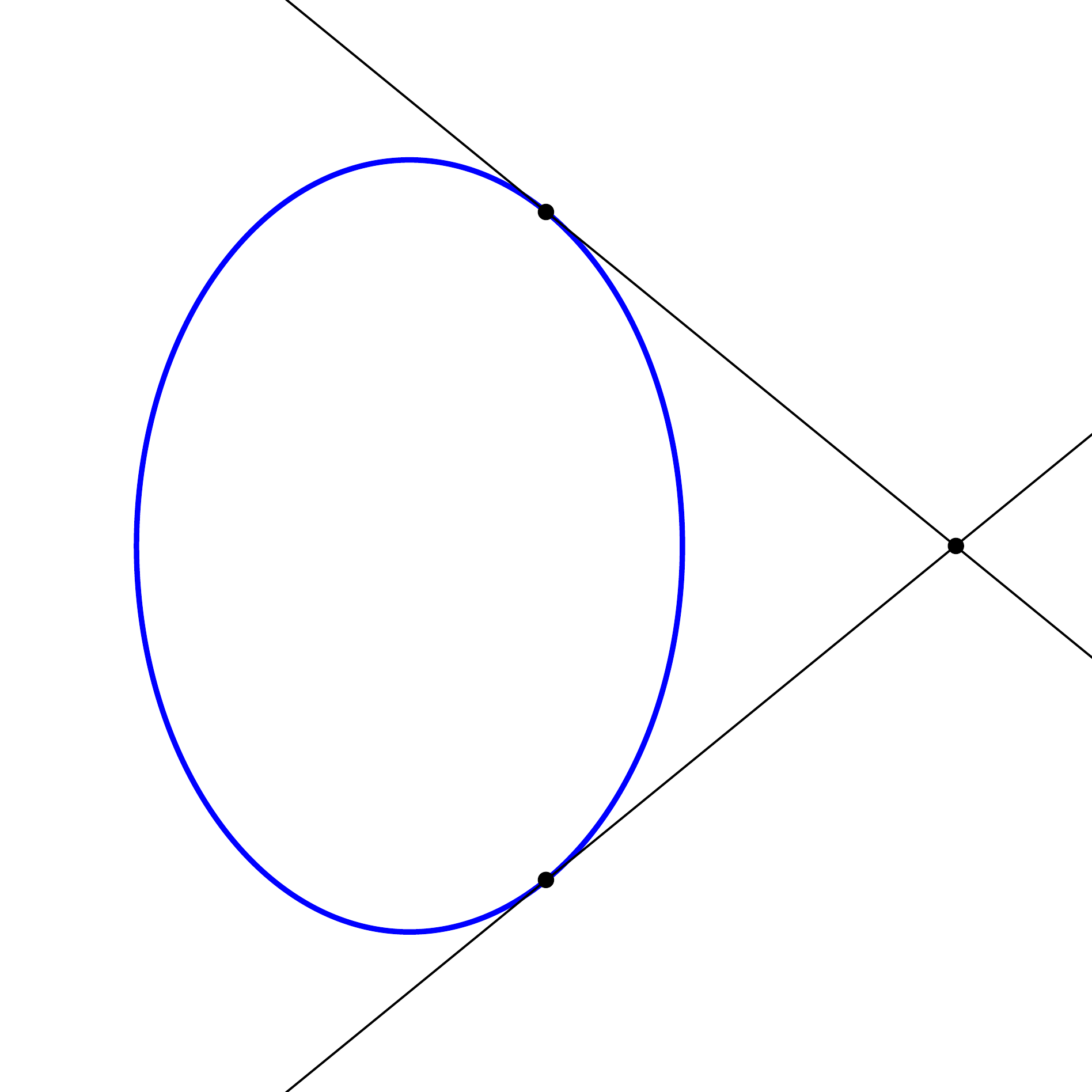}}
      \put(10,110){$C$}
      \put(85,135){$x$}
      \put(115,110){$T_xC$}
      \put(115,45){$T_yC$}
      \put(85,25){$y$}
      \put(140,70){$p$}
    \end{picture}
\caption{Polar locus of an ellipse.\label{fig:curvepolar}}
\end{figure}

Polar loci, also known as polar varieties, can be generalized to
varieties of higher dimension and play an important role in
applications of non-linear algebra. Examples include real equation
solving \cite{Bank2001}, computational complexity
\cite{Burgisser2007}, computing invariants
\cite{DiRocco2011,BATES2013493,EJP,DEP2017}, Euclidean distance degree
\cite{DHOS16} and optimization \cite{SafeyElDin}.

In this paper we use polar varieties to count bottlenecks. This is
done in the complex projective setting. For a smooth projective
variety $X \subset \PP^n$ polar loci are defined using the projective
tangent space $\mathbb{T}_xX \subset \PP^n$ at points $x \in
X$. Consider first the case where $X$ is a smooth hypersurface defined
by a homogeneous polynomial $f \in \CC[x_0,\dots,x_n]$ and let $x \in
X$. Then the hyperplane $\mathbb{T}_xX \subset \PP^n$ is defined by
the equation $\sum_{i=0}^n x_i\frac{\partial f}{\partial
  x_i}(x)=0$. In general, if $X$ is a smooth variety defined by an
ideal generated by homogeneous polynomials $f_1,\dots,f_k \in
\CC[x_0,\dots,x_n]$, then $\mathbb{T}_xX \subseteq \PP^n$ is the
subspace defined by the kernel of the Jacobian matrix
$\{\frac{\partial f_i}{\partial x_j}(x)\}_{i,j}$.

For a smooth surface $X \subset \PP^3$ we have two polar
varieties. Let $p \in \PP^3$ be a general point and $l \subset \PP^3$
a general line. Then $P_1(X,p)$ is the set of points $x$ such that the
projective tangent plane $\mathbb{T}_xX \subset \PP^3$ contains
$p$. This is a curve on $X$. Similarly, $P_2(X,l)=\{x \in X:l
\subseteq \mathbb{T}_xX\}$, which is finite. We also let
$P_0(X)=X$. More generally, an $m$-dimensional variety has $m+1$ polar
varieties defined by exceptional tangent loci as follows. Let $X
\subset \PP^n$ be a smooth variety of dimension $m$. For $j=0,\dots,m$
and a general linear space $V \subseteq \PP^n$ of dimension $n-m-2+j$
we define the polar locus \[P_j(X,V)=\{x \in X: \dim{\mathbb{T}_xX
  \cap V} \geq j-1\}.\] If $X$ has codimension $1$ and $j=0$, then $V$
is the empty set using the convention $\dim{\emptyset}=-1$. By
\cite{F98} Example 14.4.15, $P_j(X,V)$ is either empty or of pure
codimension $j$.

In order to link bottlenecks and polar varieties we employ the tools
of intersection theory and pass from polar varieties to polar
classes. For each polar variety $P_j(X,V)$ there is a corresponding polar
class $[P_j(X,V)]=p_j$ which represents $P_j(X,V)$ up to
\emph{rational equivalence}. For example, $P_j(X,V)$ represents the
same polar class $p_j,$ independently of the general choice of linear
space $V$. In a similar manner, any subvariety $Z \subset \PP^n$ has a
corresponding rational equivalence class $[Z]$. We refer to Fulton's
book \cite{F98} for background on intersection theory. For more
details on polar classes see for example \cite{P78,P15} and
\cite[Example 14.4.15]{F98}. An important point is that there is a
well defined multiplication of polar classes corresponding to
intersection of polar varieties. This means that $p_ip_j=[P_i(X,V)
  \cap P_j(X,W)]$ for $0 \leq i,j \leq m$. Here $V \subset \PP^n$ and
$W \subset \PP^n$ are general linear spaces of dimension $n-m-2+i$ and
$n-m-2+j$, respectively. To express the number of bottlenecks of a
variety in terms of polar classes we also need the notion of {\it
  degree of a class}. If $Z \subset \PP^n$ is a subvariety, $\deg{Z}$
is the number of points of $Z\cap L$, where $L \subset \PP^n$ is a
general linear space of dimension $n-\dim{Z}$. For the class $[Z]$ we
let $\deg{[Z]}=\deg{Z}$.

\subsection{Results}
Let $X \subset \CC^n$ be a smooth variety and consider the closure
$\bar{X} \subset \PP^n$ in projective space. For the purpose of
counting bottlenecks we introduce the \emph{bottleneck degree} of an
algebraic variety. Under suitable genericity assumptions (see
\defref{def:bnregular}), the bottleneck degree coincides with the
number of bottlenecks.

The orthogonality relation on $\PP^n$ is defined via the
\emph{isotropic quadric} $Q \subset \PP^n$ given in homogeneous
coordinates by $\sum_0^n x_i^2=0$. Varieties which are tangent to $Q$
are to be considered degenerate in this context and we say that a
smooth projective variety is in \emph{general position} if it
intersects $Q$ transversely.

Our main result, \thmref{thm:misc}, is a proof that the bottleneck
degree of a smooth variety $\bar{X} \subset \PP^n$ in general position
can be computed via the polar classes $p_0,\dots,p_m.$ The arguments
in the proof directly give an algorithm for expressing the bottleneck
degree in terms of polar classes. We have implemented this algorithm
in Macaulay2 \cite{M2} and the script is available at
\cite{bnscript}. We give the formula for projective curves, surfaces
and threefolds, with the following notation: $h$ denotes the
hyperplane class in the intersection ring of $\bar{X}$,
$d=\deg{\bar{X}}$ and $\epsilon_i=\sum_{j=0}^{m-i} \deg{p_j}$. We also
use $\BND{\bar{X}}$ to denote the bottleneck degree of $\bar{X}$.
\begin{itemize}
\item[] Curves in $\PP^2$: \[\BND{\bar{X}}=d^4-4d^2+3d.\]
\item[] Curves in $\PP^3$: \[\BND{\bar{X}} = \epsilon_0^2 + d^2-\deg{2h+5p_1}.\]
\item[] Surfaces in $\PP^5$: \[\BND{\bar{X}} = \epsilon_0^2 + \epsilon_1^2 + d^2-\deg{3h^2+6hp_1+12p_1^2+p_2}.\]
\item[] Threefolds in $\PP^7$: \[\BND{\bar{X}} = \epsilon_0^2 + \epsilon_1^2 + \epsilon_2^2+d^2-\deg{4h^3+11h^2p_1+4hp_1^2+24p_1^3+2hp_2-12p_1p_2+17p_3}.\]
\end{itemize}
Notice that $\epsilon_0=\deg{p_0}+\dots+\deg{p_m}$ is equal to the
Euclidean Distance Degree of the variety.

Now consider the smooth affine variety $X \subset \CC^n \subset \PP^n$
and let $H_{\infty} = \PP^n \setminus \CC^n$ be the hyperplane at
infinity. The formulas for projective varieties above have to be
modified to yield the bottleneck degree $\BND{X}$ of the affine
variety $X$. Namely, there is a contribution to $\BND{\bar{X}}$ from
the hyperplane section $X \cap H_{\infty}$ at infinity. More
precisely, we show in \propref{prop:affine}
that \[\BND{X}=\BND{\bar{X}}-\BND{\bar{X} \cap H_{\infty}}.\] Here we
have assumed that $X \subset \CC^n$ is in general position in the
following sense: $\bar{X}$ and $X \cap H_{\infty}$ are smooth and in
general position. In the case of a plane curve $X \subset \CC^2$ in
general position the hyperplane section $\bar{X}\cap H_{\infty}$
consists of $d$ points on the line at infinity. This results in
\begin{equation} \label{eq:introaffinecurve}
\BND{X}=d^4-4d^2+3d-d(d-1)=d^4-5d^2+4d.
\end{equation}

We end this introduction with an example illustrating the above
formula for affine curves $X \subset \CC^2$. Before looking at a
concrete example it is worth pointing out that by convention
bottlenecks are counted as ordered pairs $(x,y) \in X \times X$. Since
$(y,x)$ is also a bottleneck if $(x,y)$ is a bottleneck, each
unordered bottleneck pair contributes twice to the bottleneck degree.
\begin{ex*}
Consider the Trott curve $X \subset \mathbb{C}^2$ defined by the equation
\[ 144(x_1^4+x_2^4)-225(x_1^2+x_2^2)+350x_1^2x_2^2+81 .\]
\begin{figure}[htb]
  \centering
  \includegraphics[trim={0pt 0pt 0pt 30pt},clip,width=130pt]{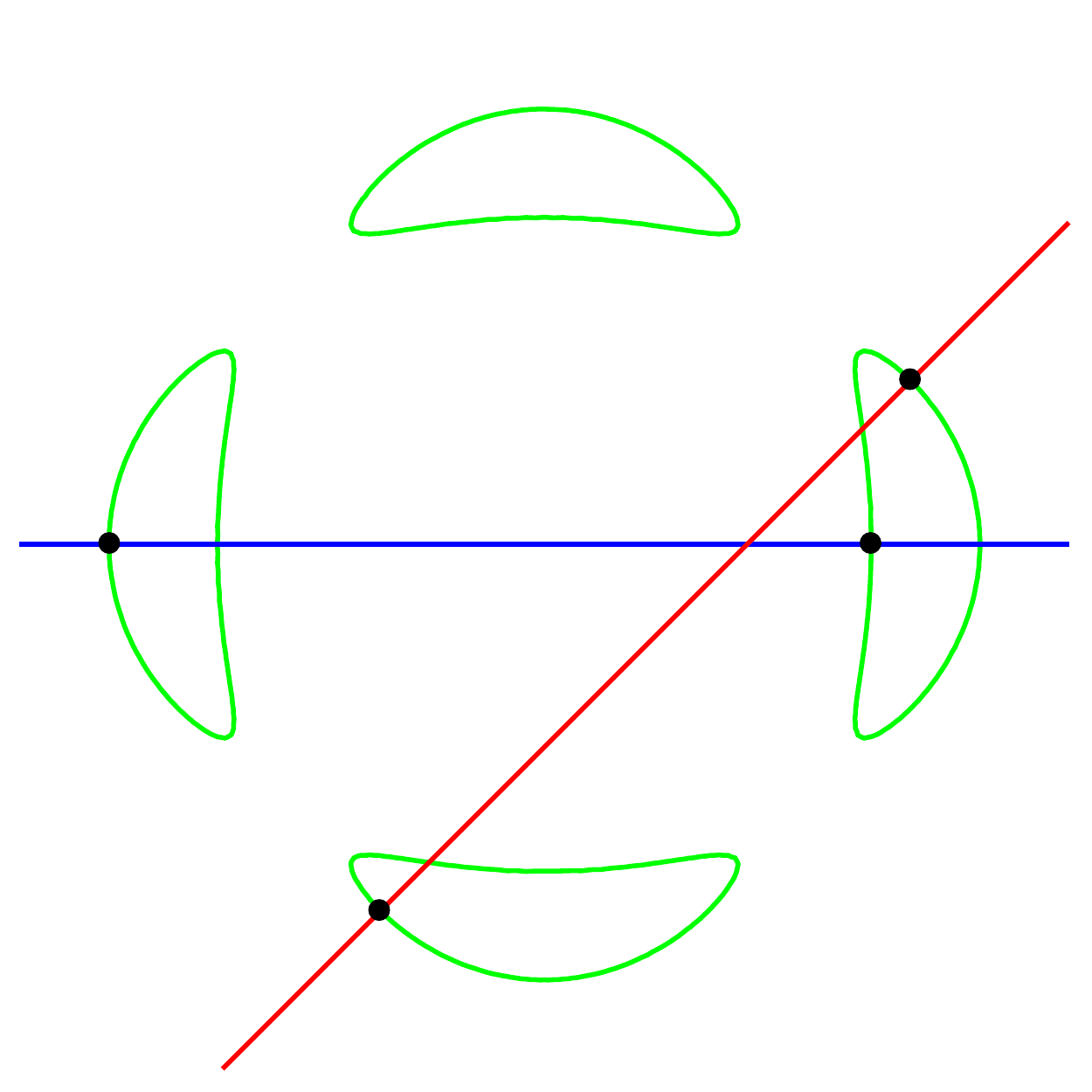}
  \caption{\label{fig:trottcurve} The quartic Trott curve depicted
    with two bottleneck pairs and their corresponding normal lines.}
\end{figure}
This nonsingular quartic curve is notable because all 28 bitangents
are real.

The bottleneck pairs $\{ (x_1,x_2), (y_1,y_2) \}$ are the off-diagonal
solutions to the following set of four equations, which are the
equations of the bottleneck ideal described in (\ref{eq:realbn}). The
first two imply that each point is on the curve and the second two
imply that each point is on the normal line to the curve at the other
point:
\begin{gather*}
144(x_1^4+x_2^4)-225(x_1^2+x_2^2)+350x_1^2x_2^2+81=0 \\
144(y_1^4+y_2^4)-225(y_1^2+y_2^2)+350y_1^2y_2^2+81=0 \\ 
x_1(-576x_1^2-700x_2^2+450)(y_2-x_2)=x_2(576x_2^2+700x_1^2-450)(x_1-y_1) \\
y_1(-576y_1^2-700y_2^2+450)(x_2-y_2)=y_2(576y_2^2+700y_1^2-450)(y_1-x_1). 
\end{gather*}
For a general enough affine plane curve of degree 4,
(\ref{eq:introaffinecurve}) gives a bottleneck degree of 192. This is
in fact the number of bottlenecks of the Trott curve. It was verified
in Macaulay2 by creating the ideal of the four equations above and
then saturating to remove the diagonal.

In this example, the 192 bottlenecks correspond to 192/2=96 bottleneck
pairs.  In particular, the real part of the Trott curve intersects the
$x$- and $y$-axis each 4 times and in each case the relevant axis is
the normal line to the curve at the intersection, leading to six
bottleneck pairs on each axis.
\end{ex*}

The paper is naturally divided between the treatment of the projective
case (Section \ref{proj}) and the affine case (Section
\ref{affine}). Section \ref{sec:examples} provides a small library of
examples.

\subsection{Acknowledgements}
This work started at KTH, Stockholm and was concluded during the three
authors' stay at ICERM, Providence. We are very grateful to ICERM for
generous financial support and for providing a stimulating working
environment. The first two authors' work and the work-visits to KTH
were supported by a VR grant [NT:2014-4763]. The second author was
supported in part by a research grant (15334) from VILLUM
FONDEN. This material is based upon work supported by the National Science Foundation Graduate Research Fellowship Program under Grant No. DGE 1752814. Any opinions, findings, and conclusions or recommendations expressed in this material are those of the authors and do not
necessarily reflect the views of the National Science Foundation. Mathematica \cite{WolframMathematica} was used for
experimentation and creating images. We would like to thank the
referees for good suggestions and valuable comments.

\section{Projective varieties}\label{proj}

\subsection{Notation and background in intersection theory} \label{sec:notation}
Below we introduce the Chow group of a subscheme of complex projective
space $\PP^n$ and present the double point formula from
intersection theory. The reason for considering schemes and not only
algebraic varieties is that isolated bottlenecks are counted with
multiplicity and similar considerations should be made for higher
dimensional bottleneck components. Specifically, the double point
class defined below is a push forward of the double point scheme and
the latter carries multiplicity information. In the end we only study
bottlenecks on algebraic varieties and little is lost if the reader
wishes to think of varieties in place of schemes.

The notation used in this paper will closely follow that of Fulton's
book \cite{F98}. Let $X \subseteq \PP^n$ be a closed
$m$-dimensional subscheme. We use $A_k(X)$ to denote the group of
$k$-cycles on $X$ up to rational equivalence and $A_*(X) =
\bigoplus_{k=0}^mA_k(X)$ denotes the Chow group of $X$. For a
subscheme $Z \subseteq X$ we have an associated cycle class $[Z] \in
A_*(X)$. Also, for a zero cycle class $\alpha \in A_0(X)$ we have the
notion of degree, denoted $\deg{\alpha}$, which counts the number of
points with multiplicity of a 0-cycle representing $\alpha$.

Suppose now that $X \subseteq \PP^n$ is a smooth variety of dimension
$m$. In this case we will also consider the intersection product on
$A_*(X)$ which makes it into a ring. For $\alpha,\beta \in A_*(X)$ we
denote their intersection product by $\alpha \beta$ or $\alpha \cdot
\beta$. Now let $\alpha \in A_k(X)$ with $k>0$ and consider the
hyperplane class $h \in A_{m-1}(X)$ induced by the embedding $X
\subseteq \PP^n$. In this paper, we define
$\deg{\alpha}=\deg{h^k\alpha}$. This means that if $\alpha$ is
represented by a subvariety $Z \subseteq X$, then $\deg{\alpha}$ is
the degree of $Z$. For a cycle class $\alpha \in A_*(X)$, we will use
$(\alpha)_k$ to denote the homogeneous piece of $\alpha$ of
codimension $k$, that is $(\alpha)_k$ is the projection of $\alpha$ to
$A_{m-k}(X)$. Finally, for $i=0,\dots,m$, $c_i(T_X)$ denotes the
$i$-th Chern class of the tangent bundle of $X$ and
$c(T_X)=c_0(T_X)+\dots +c_m(T_X)$ denotes the total Chern class.

Now let $X$ and $Y$ be subschemes of projective space. A map $f:X \to
Y$ gives rise to a push forward group homomorphism $f_*:A_*(X) \to
A_*(Y)$ and if $X$ and $Y$ are smooth varieties we also have a
pull-back ring homomorphism $f^*:A_*(Y) \to A_*(X)$.

Let $f:X \to Y$ be a morphism of smooth projective varieties. Let $x
\in A_k(X)$, $y \in A_l(Y)$ satisfy $k+l=\dim{Y}$. By the projection
formula \cite[Proposition 8.3 (c)]{F98}, $f_*(f^*(y)\cdot x) = y \cdot
f_*(x)$. In particular $\deg{y\cdot f_*(x)}=\deg{f_*(f^*(y)\cdot
  x)}=\deg{f^*(y)\cdot x}$. This relation is used many times in
the sequel.

Now let $f:X \to Y$ be a map of smooth projective varieties with
$\dim{X}=k$ and $\dim{Y}=2k$. Let $f \times f:X \times X \to Y \times
Y$ be the induced map, let $\Bl{\Delta_X}{X\times X}$ be the blow-up
of $X \times X$ along the diagonal $\Delta_X \subset X \times X$ and
let $bl:\Bl{\Delta_X}{X\times X} \to X \times X$ be the blow-up
map. Consider the map $h=(f\times f) \circ bl:\Bl{\Delta_X}{X\times
  X} \to Y \times Y$ and the inverse image scheme $h^{-1}(\Delta_Y)$
of the diagonal $\Delta_Y \subset Y \times Y$. Then the exceptional
divisor $bl^{-1}(\Delta_X)$ is a subscheme of $h^{-1}(\Delta_Y)$ and
its residual scheme in $h^{-1}(\Delta_Y)$ is called the \emph{double
  point scheme} of $f$ and is denoted $\tilde{D}(f)$. The exceptional
divisor $bl^{-1}(\Delta_X)$ may be interpreted as the projectivized
tangent bundle $\PP(T_X)$. The support of the double point scheme
$\tilde{D}(f)$ consists of the pairs of distinct points $(x,y) \in
X\times X \subset \Bl{\Delta_X}{X \times X} \setminus
bl^{-1}(\Delta_X)$ such that $f(x)=f(y)$ together with the tangent
directions in $\PP(T_X)$ where the differential $df:T_X \rightarrow T_Y$
vanishes, see \cite[Remark 14]{L78}. There is also an associated
residual intersection class $\bar{\mathbb{D}}(f) \in
A_0(\tilde{D}(f))$ defined in \cite[Theorem 9.2]{F98}. If
$\tilde{D}(f)$ has dimension 0, as expected, then $\bar{\mathbb{D}}(f)
= [\tilde{D}(f)]$. Let $\eta:\tilde{D}(f) \to X$ be the map induced by
$bl$ and the projection $X \times X \to X$ onto the first factor. Then
the \emph{double point class} $\mathbb{D}(f) \in A_0(X)$ is defined by
$\mathbb{D}(f)=\eta_*(\bar{\mathbb{D}}(f))$. By the double point
formula, \cite[Theorem 9.3]{F98}, \[\mathbb{D}(f) =
f^*f_*[X]-(c(f^*T_{Y})c(T_X)^{-1})_k.\]

\subsection{The conormal variety}

Let $X \subset \PP^n$ be a smooth variety of dimension $m$. Recall
that $H^0(X, {\mathcal O}_X(1))\cong \CC^{n+1}$ is the vector space
parameterizing the hyperplane sections of the embedding $X \subseteq
\PP^n \cong \PP(H^0(X, {\mathcal O}_X(1))).$ Consider the surjective
linear map: \[{\rm jet}_x: H^0(X, {\mathcal O}_X(1))\to H^0( {\mathcal
  O}_X(1)\otimes {\mathcal O}_X/m_x^2)\cong\CC^{m+1},\] where $m_x$ is
the maximal ideal at $x.$ Roughly speaking this map assigns to a
global section $s$ the $(m+1)$-tuple $(s(x), \ldots, \frac{\partial
  s}{\partial x_i}(x),\ldots),$ where $(x_1,\ldots,x_m)$ is a system
of coordinates around $x.$ We also have that
\[{\mathbb T}_xX=\PP(\im{{\rm jet}_x})\cong\PP^m.\]
Let $N_{X/\PP^n}$ be the normal bundle of $X$ in $\PP^n$ and let
$N_{X/\PP^n}^\vee$ be its dual.  The fibers of the dual normal bundle
at $x$ are given by the kernel of the map ${\rm jet}_x:$ $\ker({\rm
  jet}_x)\cong {N_{X/\PP^n}^\vee}_x \otimes {\mathcal O}_X(1)_x.$ The
projective tangent spaces at points $x\in X$ glue together to form the
first jet bundle ${\mathbb J}$ with fiber  ${\mathbb J}_x=H^0( {\mathcal O}_X(1)\otimes
{\mathcal O}_X/m_x^2),$ inducing the exact sequence of vector bundles:
\begin{equation}\label{eq:jet}
0\to N_{X/\PP^n}^\vee \otimes {\mathcal O}_X(1)\to X\times H^0( {\mathcal O}_X(1))\to {\mathbb J}\to 0
\end{equation}

The projectivized bundle of the conormal bundle is called the {\it
  conormal variety}:
$$\mathcal{C}_X=\PP(N_{X/\PP^n}^{\vee})\cong \PP(N_{X/\PP^n}^\vee
\otimes {\mathcal O}_X(1))\subset \PP^n \times (\PP^n)^*$$ where
$\PP(N_{X/\PP^n}^{\vee})$ denotes the projectivized conormal bundle of
$X$ in $\PP^n,$ see \cite[Example 3.2.21]{F98} for more details.  From
the exact sequence (\ref{eq:jet}) it follows that the conormal variety
consists of pairs of points $x \in X$ and hyperplanes in $\PP^n$ that
contain the projective tangent space ${\mathbb T}_xX$.

\subsection{Bottleneck degree}

Let $X \subset \PP^n$ be a smooth variety of dimension $m<n$ and
consider the conormal variety $\mathcal{C}_X=\PP(N_{X/\PP^n}^{\vee})
\subset \PP^n \times (\PP^n)^*$ introduced above.  We use
$\mathcal{O}(1)$ to denote the dual of the tautological line bundle on
$\mathcal{C}_X$, see \cite[ Appendix B.5.1 and B.5.5]{F98}, and
$\cOO=c_1(\mathcal{O}(1))$ denotes the first Chern class of
$\mathcal{O}(1)$. Also, let $\pi: \mathcal{C}_X \rightarrow X$ be the
projection. Note that $\dim{\mathcal{C}_X}=n-1$.

\begin{rem} \label{rem:relations}
  In the sequel we will compute the degrees of zero cycle classes in
  $A_0(\PP(N_{X/\PP^n}^{\vee}))$.  By \cite[ Theorem 3.3 (b)]{F98}, $A_0(X)
  \cong A_0(\PP(N_{X/\PP^n}^{\vee}))$ via the map $\alpha \mapsto
  \cOO^{n-m-1}\pi^*\alpha$. This means that every element of
  $A_0(\PP(N_{X/\PP^n}^{\vee}))$ can be written uniquely in the form
  $\cOO^{n-m-1}\pi^*\alpha$ where $\alpha \in A_0(X),$ leading to a
  degree formula:
  \[\deg{\cOO^{n-m-1}\pi^*\alpha}=\deg{\alpha}.\]
   Also, by \cite[Remark 3.2.4]{F98} 
\begin{equation} \label{eq:canonical}
  \cOO^{n-m}+c_1(\pi^*N_{X/\PP^n}^{\vee})\cOO^{n-m-1}+\dots+c_{n-m}(\pi^*N_{X/\PP^n}^{\vee})=0.
\end{equation}
  Hence, given a zero cycle class $Z \in A_0(\PP(N_{X/\PP^n}^{\vee}))$
  of the form $Z=\cOO^i\pi^*\beta$ where $i>n-m-1$ and $\beta \in
  A_{i-(n-m-1)}(X)$ we may use \eqref{eq:canonical} to write $Z$ as
  $\cOO^{n-m-1}\pi^*\alpha$ for some $\alpha \in A_0(X)$. More
  generally, consider a 0-cycle class $Z \in
  A_0(\PP(N_{X/\PP^n}^{\vee}))$ which is a polynomial in $\xi$ and
  pull-backs of classes on $X$,
  $Z=\sum_{i=0}^l\cOO^i\pi^*\beta_i$. Then $\cOO^i\pi^*\beta_i=0$ for
  $i < n-m-1$ and $\beta_i \in A_{i-(n-m-1)}(X)$ for $i\geq
  n-m-1$. Again we can use the relation \eqref{eq:canonical} to write
  $Z$ as $\cOO^{n-m-1}\pi^*\alpha$ for some $\alpha \in A_0(X)$. This
  may be done in practice by applying the function
  \verb|pseudoRemainder| in Macaulay2 \cite{M2} to $Z$ and the left
  hand side of \eqref{eq:canonical}. We will make use of this to
  compute bottleneck degrees in \algref{alg:Bmn}.
\end{rem}
  
We will consider $\mathcal{C}_X$ as a subvariety of $\PP^n \times
\PP^n$ as follows. Fix coordinates on $\PP^n$ induced by the standard
basis of $\CC^{n+1}$. Then identify $\PP^n$ with $(\PP^n)^*$ via the
isomorphism $L:\PP^n \to (\PP^n)^*$ which sends a point
$(a_0,\dots,a_n) \in \PP^n$ to the hyperplane $\{(x_0,\dots,x_n) \in
\PP^n: a_0x_0+\dots+a_nx_n=0\}$. Define $a \perp b$ by $\sum_{i=0}^n
a_ib_i=0$ for $a=(a_0,\dots,a_n), (b_0,\dots,b_n) \in \PP^n$. For a
point $p \in X$ we denote by $(\mathbb{T}_pX)^{\perp}$ the orthogonal
complement of the projective tangent space of $X$ at $p$. The span
$\langle p, (\mathbb{T}_pX)^{\perp} \rangle$ of $p$ and
$(\mathbb{T}_pX)^{\perp}$ is called the Euclidean normal space of $X$
at $p$ and is denoted $N_pX$. The Euclidean normal space is
intrinsically related to the conormal variety as:
\[\mathcal{C}_X=\{(p,q)
\in \PP^n\times \PP^n: p \in X, q \in (\mathbb{T}_pX)^{\perp}\}.\]

\begin{defn}\label{gen}
We say that a smooth variety $X \subset \PP^n$ is in \emph{general
  position} if $\mathcal{C}_X$ is disjoint from the diagonal $\Delta
\subset \PP^n \times \PP^n$.
\end{defn}

Let $Q \subset \PP^n$ be the \emph{isotropic quadric}, which is
defined by $\sum_{i=0}^n x_i^2=0$. If $p \in X \cap Q$ is that such
that $\mathbb{T}_pX \subseteq \mathbb{T}_pQ$, then $(p,p) \in
\mathcal{C}_X$. Conversely, if $(p,p) \in \mathcal{C}_X$, then $p \in
X \cap Q$ and $\mathbb{T}_pX \subseteq \mathbb{T}_pQ$. In other words,
$X$ is in general position if and only if $X$ intersects the isotropic
quadric transversely.

Suppose that $X$ is in general position. We then have a map
\begin{equation} \label{eq:normallinemap}
f:\mathcal{C}_X \rightarrow \gr{2}{n+1} : (p,q) \mapsto \langle p,q
\rangle,
\end{equation}
from $\mathcal{C}_X$ to the Grassmannian of lines in $\PP^n.$ The map
sends a pair $(p,q)$ to the line spanned by $p$ and $q$. For the
remainder of the paper, $f$ will be used to denote this map associated
to a variety $X$. To simplify notation we will also let
$G=\gr{2}{n+1}$.

Note that for $p \in X$, the map $f$ restricted to the fiber
$\{(p',q') \in \mathcal{C}_X:p'=p\}$ parameterizes lines in the
Euclidean normal space $N_pX$ passing through $p$.

\begin{ex}
In the case where $X \subset \PP^n$ is a smooth hypersurface,
$\mathcal{C}_X \cong X$ via the projection on the first factor of
$\PP^n \times (\PP^n)^*$. Consider a general curve $X \subset \PP^2$
of degree $d$ defined by a polynomial $F \in \CC[x,y,z]$. For $u \in
\{x,y,z\}$, let $F_u=\frac{\partial F}{\partial u}$. In this case
$G=(\PP^2)^*$ and the map $f:X \to (\PP^2)^*$ defined above is given
by $(x,y,z) \mapsto (yF_z-zF_y,zF_x-xF_z,xF_y-yF_x)$. Note that
$f(p)=N_pX$ is the Euclidean normal line to $X$ at $p$.
\end{ex}

Returning to a smooth $m$-dimensional variety $X \subset \PP^n$ in
general position, consider the projection $\eta: \mathcal{C}_X \times
\mathcal{C}_X \to X \times X$ and the incidence correspondence \[I(X)
= \eta(\{(u,v) \in \mathcal{C}_X\times \mathcal{C}_X: f(u)=f(v)\}).\]
Pairs $(x,y) \in I(X) \subset X\times X$ with $x\neq y$ are called
\emph{bottlenecks} of $X$. The following lemma relates this definition
of bottlenecks to the one given for affine varieties in
\secref{sec:intro}. For $x \in X$, recall the definition of the
Euclidean normal space $N_xX=\langle x, (\mathbb{T}_xX)^{\perp}
\rangle$, where $\mathbb{T}_xX$ denotes the projective tangent space
of $X$ at $x$.
\begin{lemma}
Let $X \subset \PP^n$ be a smooth variety in general position. For a
pair of distinct points $x,y \in X$, $(x,y)$ is a bottleneck if and
only if $y \in N_xX$ and $x \in N_yX$.
\end{lemma}
\begin{proof}
By definition $(x,q) \in \mathcal{C}_X \subset \PP^n \times \PP^n$ if
and only if $x \in X$ and $q \in (\mathbb{T}_xX)^{\perp}$. Hence, for $(x,q)
\in \mathcal{C}_X$, the line $\langle x,q \rangle$ is contained in
$N_xX$. Now, if $(x,y) \in X \times X$ is a bottleneck, then $(x,q),
(y,q') \in \mathcal{C}_X$ for some $q,q' \in \PP^n$ with $\langle x, q
\rangle = \langle y, q' \rangle$. Hence $y \in \langle x, q \rangle
\subseteq N_xX$. In the same way $x \in N_yX$. To see the converse let
$x,y \in X$ be distinct points such that $y \in N_xX$ and $x \in
N_yX$. Since $y \in N_xX$, $y \in \langle x, q \rangle$ for some $q
\in (\mathbb{T}_xX)^{\perp}$. Then $(x,q) \in \mathcal{C}_X$ and $q \neq x$
since $X$ is in general position. This implies that $\langle x, y
\rangle=\langle x,q \rangle$. In the same way, $x \in N_y X$ implies
that $(y,q') \in \mathcal{C}_X$ for some $q' \in \PP^n$ with $\langle
x, y \rangle=\langle y,q' \rangle$. Since $\langle x,q \rangle =
\langle y,q' \rangle$, $(x,y)$ is a bottleneck.
\end{proof}

Applying the double point formula to the map $f$ we
obtain \[\mathbb{D}(f) =
f^*f_*[\mathcal{C}_X]-(c(f^*T_{G})c(T_{\mathcal{C}_X})^{-1})_{n-1},\]
where $\mathbb{D}(f)$ is the double point class of $f$.

\begin{defn} \label{def:bndegree}
  Let $X \subset \PP^n$ be a smooth variety in general position. We
  call $\deg{\mathbb{D}(f)}$ the \emph{bottleneck degree} of $X$ and
  denote it by $\BND{X}$.
\end{defn}

The bottleneck degree is introduced to count bottlenecks on $X$ but
there are some issues that need to be considered. The first issue is
that there might be higher dimensional components worth of
bottlenecks. In this case the bottleneck degree assigns multiplicities
to these components which contribute to the bottleneck degree. We will
not pursue this aspect of bottlenecks in this paper even though it is
an interesting topic. Consider now a smooth variety $X \subset \PP^n$
in general position with only finitely many bottlenecks. As mentioned
in \secref{sec:notation}, the double point scheme of $f$ contains not
only bottlenecks but also the tangent directions in
$\PP(T_{\mathcal{C}_X})$ where the differential of $f$ vanishes. This
motivates the following definition of \emph{bottleneck regular}
varieties. As we shall see in \propref{prop:bn-number}, the bottleneck
degree is equal to the number of bottlenecks counted with multiplicity
in this case.

\begin{defn}\label{def:bnregular}
  We will call a smooth variety $X \subset \PP^n$ \emph{bottleneck
    regular} (BN-regular) if
\begin{enumerate}
\item $X$ is in general position,
\item $X$ has only finitely many bottlenecks and
\item the differential $df_p:T_p\mathcal{C}_X \to T_{f(p)}G$ of the map $f$ has
  full rank for all $p \in \mathcal{C}_X$.
\end{enumerate}
\end{defn}

\begin{prop} \label{prop:ambient}
  Assume $X$ is BN regular. Let $X \subset \PP^a \subseteq \PP^b$ be a smooth variety where
  $\PP^a \subseteq \PP^b$ is a coordinate subspace. If $X$ is in
  general position with respect to $\PP^a$ then $X$ is in general
  position with respect to $\PP^b$ and the bottleneck degree is
  independent of the choice of ambient space.
\end{prop}
\begin{proof}
For $c=a,b$, let $\mathcal{C}^c$ denote the conormal variety with
respect to the embedding $X \subset \PP^c$. The embedding $\PP^a
\subseteq \PP^b$ induces an embedding $\mathcal{C}^a \subseteq
\mathcal{C}^b$. Similarly for $c=a,b$, let $f_c:\mathcal{C}^c \to
\gr{2}{c+1}$ be the map given by $(p,q) \mapsto \langle p,q \rangle$
and let $\Delta_c \subset \PP^c \times \PP^c$ be the diagonal. Suppose
that $(p,q) \in \mathcal{C}^b \cap \Delta_b$. Since $p \in X \subset
\PP^a$, we have that $q=p \in \PP^a$ and $(p,q) \in \mathcal{C}^a
\cap \Delta_a = \emptyset$. Hence $X$ is in general position with
respect to $\PP^b$.

We will consider $\mathbb{D}(f_a)$ as a cycle class on
$\mathcal{C}^b$ via the inclusion $\mathcal{C}^a \subseteq
\mathcal{C}^b$. We will show that $\tilde{D}(f_a) =
\tilde{D}(f_b)$.  Since $X$ is BN-regular $[\tilde{D}(f_c)]= \mathbb{D}(f_c)$
for $c=a,b.$   
It follows that $\mathbb{D}(f_a)=\mathbb{D}(f_b)$,
which in turn implies that the bottleneck degree is independent of the
choice of ambient space. Note that $\tilde{D}(f_a) \subseteq
\tilde{D}(f_b)$.

We will first show that the differential $df_b:T_{\mathcal{C}^b} \to
T_G$ has full rank outside $T_{\mathcal{C}^a}$. This implies that
$\tilde{D}(f_b) \setminus \tilde{D}(f_a)$ consists of pairs $x,y \in
\mathcal{C}^b$ with $x \neq y$ and $f_b(x)=f_b(y)$. Suppose that
$x=(p,q) \in \mathcal{C}^b$ and $v \in T_x\mathcal{C}^b$ is a non-zero
tangent vector such that $(df_b)_x(v)=0$. Let $D \subset \CC$ be the
unit disk and let $P,Q:D \to \CC^{b+1} \setminus \{0\}$ be smooth
analytic curves such that the induced curve $D \to \PP^n \times \PP^n$
is contained in $\mathcal{C}^b$, passes through $x=(p,q)$ at $0 \in D$
and has tangent vector $v$ there. In other words $P(0) \in p$ and
$Q(0) \in q$ are representatives of $p$ and $q$. We need to show that
$Q(0),Q'(0) \in \CC^{a+1}$. Since $(df_b)_x(v)=0$, we have by
\cite[Example 16.1]{H92}  that $P'(0),Q'(0) \in \langle
P(0),Q(0)\rangle$. Suppose first that $P'(0)$ and $P(0)$ are
independent. Then $Q(0),Q'(0) \in \langle P(0),P'(0) \rangle$ and
$\langle P(0),P'(0) \rangle \subseteq \CC^{a+1}$ since $X \subset
\PP^a$. Now suppose that $P'(0)$ is a multiple of $P(0)$. Since $v
\neq 0$, $Q(0)$ and $Q'(0)$ are independent and $Q'(0)$ corresponds to
a point $q' \in \PP^n$. That $P(0)$ and $P'(0)$ are dependent implies
that $(p,q') \in \mathcal{C}^b$. Moreover, $P(0) \in \langle Q(0),
Q'(0) \rangle$ by above and hence $p \in \langle q,q' \rangle$. It
follows that $(p,p) \in \mathcal{C}^b$, which contradicts that $X$ is
in general position.

Now let $(x,y) \in \tilde{D}(f_b)$ with $x \neq y$ and
$f_b(x)=f_b(y)$. If $x,y \in \mathcal{C}^a$ then $(x,y) \in
\tilde{D}(f_a)$ so assume that $x \notin \mathcal{C}^a$. Let
$x=(p_1,q_1)$ and $y=(p_2,q_2)$ with $(p_i,q_i) \in X \times
\PP^b$. Since $\langle p_1,q_1 \rangle=\langle p_2, q_2 \rangle$ and
because this line intersects $\PP^a$ in exactly one point $p \in \PP^a$, we
have that $p_1=p_2=p$. Moreover, $p \in \langle q_1,q_2 \rangle$ and
hence $(p,p) \in \mathcal{C}^a$ contradicting that $X$ is in general
position. This means that $\tilde{D}(f_b) \subseteq \tilde{D}(f_a)$
and hence $\tilde{D}(f_a)=\tilde{D}(f_b)$.
\end{proof}

If $X \subset \PP^n$ is BN-regular, then the double point scheme
$\tilde{D}(f)$ is finite and in one-to-one correspondence with the
bottlenecks of $X$ through the projection $\eta:\mathcal{C}_X \times
\mathcal{C}_X \to X \times X$. Using the scheme-structure of
$\tilde{D}(f)$ we assign a multiplicity to each bottleneck. With
notation as in \secref{sec:notation},
$[\tilde{D}(f)]=\bar{\mathbb{D}}(f)$ and we therefore have the
following.

\begin{prop} \label{prop:bn-number}
  If $X \subset \PP^n$ is BN-regular, then $\BND{X}$ is equal to
  the number of bottlenecks of $X$ counted with multiplicity.
\end{prop}

\begin{rem} \label{rem:chern}
Recalling the notation from above, $\mathcal{O}(1)$ denotes the dual
of the tautological line bundle on the conormal variety
$\mathcal{C}_X$, $\pi:\mathcal{C}_X \to X$ is the projection and
$\xi=c_1(\mathcal{O}(1))$. The bottleneck degree depends on the Chern
classes of $\mathcal{C}_X$ and below we shall relate these to the
Chern classes of $X$, the hyperplane class and $\xi$. By \cite[Example 3.2.11]{F98}
 we have that
$c(T_{\mathcal{C}_X})=c(\pi^*T_X)c(\pi^*N_{X/\PP^n}^{\vee} \otimes
\mathcal{O}(1))$. Since the rank of $N_{X/\PP^n}^{\vee}$ is $n-m$ we
have by \cite[Remark 3.2.3]{F98}  that
\[\begin{matrix}
c(\pi^*N_{X/\PP^n}^{\vee} \otimes \mathcal{O}(1))&=&\sum_{i=0}^{n-m}c_i(\pi^*N_{X/\PP^n}^{\vee})(1+\cOO)^{n-m-i}\\
&=&\sum_{i=0}^{n-m}(-1)^i\pi^*c_i(N_{X/\PP^n})(1+\cOO)^{n-m-i}.
\end{matrix}\]
Note also that $c_i(N_{X/\PP^n})=0$ for $i>m=\dim{X}$. Moreover, the
normal bundle $N_{X/\PP^n}$ is related to the tangent bundles $T_X$
and $T_{\PP^n}$ by the exact sequence \[0 \to T_{X} \to i^*T_{\PP^n}
\to N_{X/\PP^n} \to 0,\] where $i:X \to \PP^n$ is the inclusion. It
follows that $c(N_{X/\PP^n})=c(i^*T_{\PP^n})c(T_X)^{-1}$. Also,
$c(T_{\PP^n})=(1+H)^{n+1}$ where $H \in A_{n-1}(\PP^n)$ is the
hyperplane class.
\end{rem}

For $n-1 \geq a\geq b \geq 0$, define the Schubert class $\sigma_{a,b}
\in A_*(G)$ as the class of the locus $\Sigma_{a,b} = \{l\in G: l \cap
A \neq \emptyset, l \subset B \}$ where $A \subset B \subseteq \PP^n$
is a general flag of linear spaces with $\codim{A}=a+1$ and
$\codim{B}=b$. In the case $b=0$ we use the notation
$\sigma_{a,0}=\sigma_a$. See \cite{EH16} for basic properties of
Schubert classes. In particular we will make use of the relations
$\sigma_1^2=\sigma_{1,1}+\sigma_2$ if $n \geq 3$ and
$\sigma_{a+c,b+c}=\sigma_{c,c}\sigma_{a,b}$ for $n-1 \geq a\geq b \geq
0$ and $0\leq a+c\leq n-1$. Also, $\sigma_{n-1-i,i}\cdot
\sigma_{n-1-j,j}=0$ for $0\leq i,j \leq \left \lfloor \frac{n-1}{2}
\right \rfloor$ if $i\neq j$ and $\sigma_{n-1-i,i}^2$ is the class of
a point. In \algref{alg:Bmn} below we will need to express the total
Chern class $c(T_{G})$ of the Grassmannian as a polynomial in Schubert
classes. To do this we apply the routine \verb|chern| from the
Macaulay2 package \emph{Schubert2} \cite{Schubert2Source}.

We will recall the definition of the polar classes $p_0,\dots,p_m \in
A_*(X)$ of $X$. For a general linear space $V \subseteq \PP^n$ of
dimension $n-m-2+j$ we have that $p_j$ is the class represented by the
polar locus \[P_j(X,V)=\{x \in X: \dim{\mathbb{T}_xX \cap V} \geq
j-1\}.\] If $X$ has codimension $1$ and $j=0$, then $V$ is the empty
set using the convention $\dim{\emptyset}=-1$. By \cite[Example
14.4.15]{F98}, $P_j(X,V)$ is either empty or of pure codimension $j$ and
\begin{equation} \label{eq:polar-chern}
  p_j = \sum_{i=0}^j (-1)^i \binom{m-i+1}{j-i}h^{j-i}c_i(T_X),
\end{equation}
where $h \in A_{n-1}(X)$ is the hyperplane class. Moreover, the polar
loci $P_j(X,V)$ are reduced, see \cite{P78}. Inverting the
relationship between polar classes and Chern classes we get
\begin{equation} \label{eq:chern-polar}
  c_j(T_X) = \sum_{i=0}^j (-1)^i \binom{m-i+1}{j-i}h^{j-i}p_i.
\end{equation}

We will examine an alternative interpretation of polar classes via the
conormal variety $\mathcal{C}_X$. This will help us to determine the
class of $\mathcal{C}_X$ in $A_*(\PP^n \times \PP^n)$. Recall that the
polar loci $P_j(X,V)$ are either empty or of codimension $j$. It
follows that for a generic point $x \in P_j(X,V)$, $\mathbb{T}_xX$
intersects $V$ in exactly dimension $j-1$, that is $\dim{\mathbb{T}_xX
  \cap V} = j-1$. Let $0\leq i \leq m$ and let $\hat{V},W \subseteq
\PP^n$ be general linear spaces with $\dim{\hat{V}}=i+1$ and
$\dim{W}=n-i$. Recall the fixed isomorphism $L:\PP^n \to (\PP^n)^*$
and let $V \subset \PP^n$ be the intersection of all hyperplanes in
$L(\hat{V})$. Note that $\dim{V}=n-2-i$. Now consider the intersection
$J = \mathcal{C}_X \cap (W \times \hat{V}) \subseteq \PP^n \times
\PP^n$. Then $J$ is finite and we have the projection map $\pi_{|J}:J
\to P_{m-i}(X,V) \cap W$. Now, $\pi_{|J}$ is bijective onto
$P_{m-i}(X,V) \cap W$ because given $x \in P_{m-i}(X,V) \cap W$,
$\dim{\mathbb{T}_xX \cap V}=m-i-1$ and therefore the span of
$\mathbb{T}_xX$ and $V$ is the unique hyperplane containing
$\mathbb{T}_xX$ and $V$. Let $\alpha, \beta \in A_{2n-1}(\PP^n \times
\PP^n)$ be the pullbacks of the hyperplane class of $\PP^n$ under the
two projections and consider $[\mathcal{C}_X]$ as an element of
$A_*(\PP^n \times \PP^n)$. Then $[W \times \hat{V}] =
\alpha^{i}\beta^{n-1-i}$ and $\deg{[\mathcal{C}_X] \cdot
  \alpha^{i}\beta^{n-1-i}}=\deg{J}=\deg{p_{m-i}}$. Note that
$[\mathcal{C}_X] \cdot \alpha^{i}=0$ if $i>m$ since $\alpha$ is the
pullback of a divisor on $\PP^n$.

\begin{thm} \label{thm:misc}
Let $X \subset \PP^n$ be a smooth $m$-dimensional variety in general
position. Let $h=\pi^*(h_X) \in A_*(\mathcal{C}_X)$ where $h_X \in
A_*(X)$ is the hyperplane class and $\pi:\mathcal{C}_X \to X$ is the
projection. We use $\mathcal{O}(1)$ to denote the dual of the
tautological line bundle on $\mathcal{C}_X$ and $\xi$ to denote its
first Chern class. Also $\alpha, \beta \in A_{2n-1}(\PP^n \times
\PP^n)$ denote the pullbacks of the hyperplane class of $\PP^n$ under
the two projections. Let $k=\min\{\left \lfloor \frac{n-1}{2} \right
\rfloor,m\}$ and for $i=0,\dots,k$, put $\epsilon_i=\sum_{j=r_i}^{m-i}
\deg{p_j}$ where $r_i=\max\{0,m-n+1+i\}$. Then the following holds:
\begin{align}
  \label{eq:conormal}
  [\mathcal{C}_X] = \sum_{i=0}^{m} \deg{p_{m-i}}
  \alpha^{n-i}\beta^{1+i},\\
  \label{eq:schubert}
    f^*(\sigma_{a,b}) = \sum_{i=0}^{a-b}
h^{b+i}(\cOO-h)^{a-i},\\
  \label{eq:image}
  f_*[\mathcal{C}_X] = \sum_{i=0}^k \epsilon_i
  \sigma_{n-1-i,i},\\
  \label{eq:sos}
  \deg{f^*f_*[\mathcal{C}_X]} = \sum_{i=0}^k \epsilon_i^2.
\end{align}
Hence \[\BND{X}=\sum_{i=0}^k \epsilon_i^2-\deg{B_{m,n}},\] for some polynomial $B_{m,n}$ in the 
polar classes and the hyperplane class  of $X.$
\end{thm}
\begin{proof}
To show (\ref{eq:conormal}), note that
$\codim{\mathcal{C}_X}=2n-(n-1)=n+1$ and write $[\mathcal{C}_X] =
\sum_{i=0}^{n-1} d_i \alpha^{n-i}\beta^{1+i}$ for some $d_i \in
\ZZ$. Let $0 \leq i \leq n-1$. Because $d_i=\deg{[\mathcal{C}_X] \cdot
  \alpha^i\beta^{n-1-i}}$, it follows that:
  $$d_i=\deg{p_{m-i}}\text{ if }0 \leq i \leq m\text{ 
and }d_i=0\text{ if  }i>m.$$

Let $bl: \Bl{\Delta}{\PP^n \times \PP^n} \to \PP^n \times \PP^n$ be
the blow-up of $\PP^n \times \PP^n$ along the diagonal $\Delta \subset
\PP^n \times \PP^n$ and let $E = bl^{-1}(\Delta)$, the exceptional
divisor. The map $\PP^n \times \PP^n \setminus \Delta \to
\gr{2}{n+1},$ which sends a pair of points $(p,q)$ to the line spanned
by $p$ and $q,$ extends to a map $\gamma : \Bl{\Delta}{\PP^n \times
  \PP^n} \to \gr{2}{n+1}$, see \cite{J78}. The theorem in Appendix B
paragraph 3 of \cite{J78}, with $X=\PP^N$ in the notation used there,
states that \[\gamma^*(\sigma_a) = \sum_{i=0}^a
bl^*\alpha^ibl^*\beta^{a-i} +
\sum_{i=0}^{a-1}(-1)^{i+1}\binom{a+1}{i+2}bl^*\alpha^{a-1-i}[E]^{i+1}.\]
Consider $\mathcal{C}_X$ as a subvariety of $\Bl{\Delta}{\PP^n \times
  \PP^n}$ and let $i:\mathcal{C}_X \to \Bl{\Delta}{\PP^n \times
  \PP^n}$ be the embedding. Then $i^*bl^*\alpha=h$ and by
\cite[Example 3.2.21]{F98}, $\cOO-h=i^*bl^*\beta$. Moreover, since $X$
is in general position, $i^*[E]=0$. Using $f=\gamma \circ i$, we get
that $f^*(\sigma_a)=i^*\gamma^*(\sigma_a)=\sum_{i=0}^a
h^i(\cOO-h)^{a-i}$. In particular, $f^*(\sigma_1)=\cOO$, which proves
(\ref{eq:schubert}) in the case $n=2$. If $n \geq 3$, we have by above
that $f^*(\sigma_2)=\cOO^2 -h\cOO + h^2$. Moreover
$\sigma_{1,1}=\sigma_1^2-\sigma_2$, and hence
$f^*(\sigma_{1,1})=h(\cOO - h)$. Since $\sigma_{b,b}=\sigma_{1,1}^b$
we get $f^*(\sigma_{b,b})=h^b(\cOO -h)^b$. Finally, using
$\sigma_{a,b}=\sigma_{b,b}\sigma_{a-b}$ we get that
$f^*(\sigma_{a,b})=h^b(\cOO-h)^b\sum_{i=0}^{a-b}h^i(\cOO-h)^{a-b-i}$,
which gives (\ref{eq:schubert}).

For (\ref{eq:image}), note first that
\begin{equation} \label{eq:pullback}
\gamma^*(\sigma_{a,b}) =
\gamma^*(\sigma_{b,b})\gamma^*(\sigma_{a-b})=\sum_{i=0}^{a-b}
bl^*\alpha^{b+i}bl^*\beta^{a-i} + R,
\end{equation}
where $R=[E]\cdot \delta$ for some $\delta \in A_*(\Bl{\Delta}{\PP^n
  \times \PP^n})$. Also, $f_*[\mathcal{C}_X]=\sum_{i=0}^{s} e_i
\sigma_{n-1-i,i}$ where $e_i=\deg{f_*[\mathcal{C}_X]\cdot
  \sigma_{n-1-i,i}}$ and $s=\left \lfloor \frac{n-1}{2} \right
\rfloor$. Since $\gamma$ restricts to $f$ on $\mathcal{C}_X$,
$f_*[\mathcal{C}_X] = \gamma_*[\mathcal{C}_X]$ and
$e_i=\deg{\gamma_*[\mathcal{C}_X]\cdot
  \sigma_{n-1-i,i}}=\deg{[\mathcal{C}_X]\cdot
  \gamma^*\sigma_{n-1-i,i}}$ by the projection formula. Here
$[\mathcal{C}_X]$ denotes the class of $\mathcal{C}_X$ on
$\Bl{\Delta}{\PP^n \times \PP^n}$ and $[\mathcal{C}_X] \cdot R=0$
since $X$ is in general position. Moreover, by (\ref{eq:conormal}) we
have that $[\mathcal{C}_X]=bl^*(\sum_{l=0}^{m} \deg{p_{m-l}}
\alpha^{n-l}\beta^{l+1})$. Using (\ref{eq:pullback}) we
get \[[\mathcal{C}_X]\cdot \gamma^*\sigma_{n-1-i,i} =
bl^*(\sum_{l=0}^{m} \deg{p_{m-l}} \alpha^{n-l}\beta^{l+1})\cdot
bl^*(\sum_{j=0}^{n-1-2i} \alpha^{i+j}\beta^{n-1-i-j}).\] It follows
that $[\mathcal{C}_X]\cdot \gamma^*\sigma_{n-1-i,i} = 0$ if $i>m$. For
$i \leq m$, we get \[[\mathcal{C}_X]\cdot \gamma^*\sigma_{n-1-i,i} =
bl^*(\alpha^n\beta^n)\sum_{j=0}^{t} \deg{p_{m-(i+j)}},\] where $t =
\min\{m-i,n-1-2i\}$. Hence $e_i=0$ for $i>m$ and $e_i=\epsilon_i$
otherwise, which gives (\ref{eq:image}).

To show (\ref{eq:sos}), let $0 \leq i \leq k$ and note that by the
projection formula 
$$\epsilon_i = \deg{f_*[\mathcal{C}_X]\cdot
  \sigma_{n-1-i,i}}=\deg{[\mathcal{C}_X]\cdot
  f^*\sigma_{n-1-i,i}}=\deg{f^*\sigma_{n-1-i,i}}.$$ Hence applying
$f^*$ to (\ref{eq:image}) gives (\ref{eq:sos}).

Since the intersection ring of $\gr{2}{n+1}$ is generated by
$\sigma_{a,b}$ as a group, we may express $c(f^*T_{G})$ as a
polynomial in $\cOO$ and $h$ by (\ref{eq:schubert}). Moreover,
$c(T_{\mathcal{C}_X})$ is a polynomial in pullbacks of polar classes,
$h$ and $\cOO$ by (\ref{eq:chern-polar}) and \remref{rem:chern}. It
follows from \remref{rem:relations} that
$(c(f^*T_{G})c(T_{\mathcal{C}_X})^{-1})_{n-1}=\cOO^{n-m-1}\pi^*B_{m,n}$
for some polynomial $B_{m,n}$ in polar classes and the hyperplane
class of $X$. Also by \remref{rem:relations},
$\deg{\cOO^{n-m-1}\pi^*B_{m,n}}=\deg{B_{m,n}}$.
\end{proof}

Note that the polynomials $B_{m,n}$ in \thmref{thm:misc} only depend
on $n$ and $m$. Combining \thmref{thm:misc}, \remref{rem:chern} and
\remref{rem:relations} gives an algorithm to compute polynomials
$B_{m,n}$ as in \thmref{thm:misc}. We will now give a high level
description of this algorithm. It has been implemented in Macaulay2
\cite{M2} and is available at \cite{bnscript}.

We will use the notation in \thmref{thm:misc}. In addition, we use
$p_1,\dots,p_m$ to denote the polar classes of $X$ and $c_1,\dots,c_m$
to denote the Chern classes of $X$. Also $i:X \to \PP^n$ denotes the
inclusion and $h_X$ is the hyperplane class on $X$. The algorithm
makes use of the routines \verb|pseudoRemainder| from \cite{M2} and
\verb|chern| from \cite{Schubert2Source}.

The input to the algorithm are integers $0<m<n$ and the output is a
polynomial $B_{m,n}$ in $p_1,\dots,p_m,h_X$ such that
$(c(f^*T_{G})c(T_{\mathcal{C}_X})^{-1})_{n-1}=\cOO^{n-m-1}\pi^*B_{m,n}$.

\begin{algorithm}
\caption{Algorithm to compute polynomial $B_{m,n}$ in
  $p_1,\dots,p_m,h_X$ as in \thmref{thm:misc}} \label{alg:Bmn}
\begin{algorithmic}
  \REQUIRE Integers $0<m<n$.
  \ENSURE Polynomial $B_{m,n}$.
\STATE Invert $c(T_X)$:
  $c(T_X)^{-1}=1-\delta+\delta^2+\dots+(-1)^m\delta^m$ where
  $\delta=c(T_X)-1$.

\STATE Let $c(i^*T_{\PP^n})=(1+h_X)^{n+1}$.

\STATE Compute $c(N_{X/\PP^n})=c(i^*T_{\PP^n})c(T_X)^{-1}$.

\STATE Compute $c(\pi^*N_{X/\PP^n}^{\vee} \otimes \mathcal{O}(1)) =
  \sum_{j=0}^{n-m}(-1)^j\pi^*c_j(N_{X/\PP^n})(1+\cOO)^{n-m-j}$.

\STATE Compute
  $c(T_{\mathcal{C}_X})=c(\pi^*T_X)c(\pi^*N_{X/\PP^n}^{\vee} \otimes
  \mathcal{O}(1))$.

\STATE Invert $c(T_{\mathcal{C}_X})$:
$c(T_{\mathcal{C}_X})^{-1}=1-\delta+\delta^2+\dots+(-1)^{n-1}\delta^{n-1}$
where $\delta=c(T_{\mathcal{C}_X})-1$.

 \STATE Apply \verb|chern| to express $c(T_{G})$ as a polynomial in
 Schubert classes $\sigma_{a,b}$.

\STATE Apply the substitution (\ref{eq:schubert}) to express
  $c(f^*T_{G})=f^*c(T_{G})$ as a polynomial in $\xi$ and $\pi^*h_X$.

  \STATE Compute $(c(f^*T_{G})c(T_{\mathcal{C}_X})^{-1})_{n-1}$.

\STATE Let
$R=\cOO^{n-m}-c_1(\pi^*N_{X/\PP^n})\cOO^{n-m-1}+\dots+(-1)^{n-m}c_{n-m}(\pi^*N_{X/\PP^n})$.

  \STATE Let $P$ be the output of \verb|pseudoRemainder| applied to
  $(c(f^*T_{G})c(T_{\mathcal{C}_X})^{-1})_{n-1}$ and $R$.

  \STATE Let $\hat{B}_{m,n}$ be $P$ divided by $\xi^{n-m-1}$ and with
  $\pi^*c_1,\dots,\pi^*c_m,\pi^*h_X$ replaced by $c_1,\dots,c_m,h_X$.

  \STATE Replace $c_1,\dots,c_m$ by $p_1,\dots,p_m$ using
  (\ref{eq:chern-polar}) on $\hat{B}_{m,n}$ to acquire $B_{m,n}$.
\end{algorithmic}
\end{algorithm}

\begin{coro} \label{coro:specific-formulas}
Let $X \subset \PP^n$ be a smooth variety in general position. Let $d=\deg{X}=\deg{p_0}$, $\epsilon_i=\sum_{j=0}^{m-i} \deg{p_j}$ with
$m=\dim{X}$ and $h \in A_{m-1}(X)$  the hyperplane class. The following holds:
\begin{enumerate}
\item If $X$ is a curve in 
$\PP^2$  then  \[\BND{X}=d^4-4d^2+3d.\]
\item \label{item:space-curve} If $X $ is a curve in $\PP^3$ then \[\BND{X}=\epsilon_0^2+d^2-5\deg{p_1}-2d,\] where $\epsilon_0=d+\deg{p_1}$
is the Euclidean distance degree of $X$. 
\item If $X$ is a surface in $\PP^5$ then \[\BND{X}=\epsilon_0^2 + \epsilon_1^2 +
  d^2-\deg{3h^2+6hp_1+12p_1^2+p_2}.\]
  \item If $X$ is a  threefold in $\PP^7$ then  \[\BND{X}=\epsilon_0^2 + \epsilon_1^2 + \epsilon_2^2+
  d^2-\deg{4h^3+11h^2p_1+4hp_1^2+24p_1^3+2hp_2-12p_1p_2+17p_3}.\]
\end{enumerate}
\end{coro}
\begin{proof}
The formulas are acquired by applying \algref{alg:Bmn}, which has been
implemented in Macaulay2 \cite{M2} and is available at
\cite{bnscript}.

For illustrative purposes we will carry out the computation for curves
in $\PP^3$. By the double point formula \[\mathbb{D}(f) =
f^*f_*[\mathcal{C}_X]-(c(f^*T_{G})c(T_{\mathcal{C}_X})^{-1})_{2},\]
and using \thmref{thm:misc} we get that
$\deg{f^*f_*[\mathcal{C}_X]}=\epsilon_0^2+\epsilon_1^2=\epsilon_0^2+d^2$. Moreover,
$c(f^*T_{G})=f^*c(T_G)=1+4f^*\sigma_1+7(f^*\sigma_2+f^*\sigma_{1,1})$
and $f^*\sigma_1=\cOO$,
$f^*\sigma_2=\cOO^2-{\cOO}h+h^2=\cOO^2-{\cOO}h$
and $f^*\sigma_{1,1}=h\cOO$ by \thmref{thm:misc}.

Let $c_1$ denote the first Chern class of $X$. To compute
$c(T_{\mathcal{C}_X})$ we follow the steps of
\remref{rem:chern}. First of all
$\pi^*c(N_{X/\PP^n}^{\vee})=1-4h+\pi^*c_1$. Hence we get that
$c(\pi^*N_{X/\PP^n}^{\vee} \otimes
\mathcal{O}(1))=(1+\cOO)^2+(1+\cOO)(-4h+\pi^*c_1)$. Moreover,
by (\ref{eq:canonical}),
$\cOO^2=-\pi^*c_1(N_{X/\PP^n}^{\vee})\cOO=(4h-\pi^*c_1)\cOO$
and hence $c(\pi^*N_{X/\PP^n}^{\vee} \otimes
\mathcal{O}(1))=1+2\cOO+(-4h+\pi^*c_1)$. This means that
\[\begin{matrix} c(T_{\mathcal{C}_X})=c(\pi^*N_{X/\PP^n}^{\vee} \otimes
\mathcal{O}(1))c(\pi^*T_X)&=&(1+2\cOO-4h+\pi^*c_1)(1+\pi^*c_1)\\
&=& 1+2\cOO-4h+2\pi^*c_1+2{\cOO}\pi^*c_1.\end{matrix}\] Hence
$c(T_{\mathcal{C}_X})^{-1}=1-2\cOO+4h-2\pi^*c_1+2{\cOO}\pi^*c_1.$
It follows that:
\[(c(f^*T_{G})c(T_{\mathcal{C}_X})^{-1})_{2}=((1+4f^*\sigma_1+7(f^*\sigma_2+f^*\sigma_{1,1}))(1-2\cOO+4h-2\pi^*c_1+2{\cOO}\pi^*c_1))_2.\] Multiplying out and using the expressions for $f^*\sigma_2$ and $f^*\sigma_{1,1}$ above we get
\[\begin{matrix}
  (c(f^*T_{G})c(T_{\mathcal{C}_X})^{-1})_{2} &=&
  7f^*\sigma_2+7f^*\sigma_{1,1}+4f^*\sigma_1(-2\cOO+4h-2\pi^*c_1)+2{\cOO}\pi^*c_1
  \\ &=&
  7(\cOO^2-{\cOO}h)+7h\cOO+4\cOO(-2\cOO+4h-2\pi^*c_1)+2{\cOO}\pi^*c_1.
\end{matrix}\]
Simplifying the last expression results in the following formulas:
\[\begin{matrix}
   (c(f^*T_{G})c(T_{\mathcal{C}_X})^{-1})_{2} &=&
  -\cOO^2+16{\cOO}h-6{\cOO}\pi^*c_1 \\ &=&
  -(4h-\pi^*c_1)\cOO+16{\cOO}h-6{\cOO}\pi^*c_1 \\ &=&
  \cOO(12h-5\pi^*c_1).
  \end{matrix}\]
Finally, using \remref{rem:relations},  we get
$\deg{(c(f^*T_{G})c(T_{\mathcal{C}_X})^{-1})_{2}}=12d-5\deg{c_1}=2d+5\deg{p_1},$
since $\deg{p_1}=2d-\deg{c_1}.$ This shows the claim about $\BND{X}$
for a smooth curve $X \subset \PP^3$ in general position.

In the case of a general plane curve $X \subset \PP^2$ we have that
$\deg{c_1}=2-2g$ where $g=(d-1)(d-2)/2$ is the genus of $X$. It
follows that $\deg{p_1}=d^2-d$ and $\epsilon_0=d^2$ and
$\BND{X}=d^4-4d^2+3d$.
\end{proof}

\begin{rem} \label{rem:polys}
The formulas in \cororef{coro:specific-formulas} are given for specific
ambient dimensions $n$. For example, \cororef{coro:specific-formulas}
(\ref{item:space-curve}) is for curves in $\PP^3$ and one may ask if
the same formula is valid for curves in $\PP^4$. For the formulas
given in \cororef{coro:specific-formulas} we have checked that they are
valid for any ambient dimension $n \leq 30$ (excluding the case
$X=\PP^n$). This was done using the Macaulay2 implementation
\cite{bnscript}.

Consider now the general case of a smooth $m$-dimensional variety
$X\subset\PP^n$. Combining \algref{alg:Bmn} and \thmref{thm:misc} we
get an algorithm that, for any given $m$ and $n$, computes the
bottleneck degree of a smooth $m$-dimensional variety $X \subset
\PP^n$ in general position. The result is a formula that expresses the
bottleneck degree in terms of polar classes of $X$. Now, if we let
$n=2m+1$ we get a formula for each $m$. It is our belief that through
projection arguments one can show that this formula is in fact valid
in any ambient dimension $n>m$. Thus we conjecture that the formula in
terms of polar classes only depends on the dimension $m$.
\end{rem}

\section{The affine case}\label{affine}
In this section we define bottlenecks for affine varieties and show
how they may be counted using the bottleneck degrees of projective
varieties.

Let $X \subset \CC^n$ be a smooth affine variety of dimension
$m$. Consider coordinates $x_0,\dots,x_{n-1}$ given by the standard
basis on $\CC^n$ and the usual embedding $\CC^n \subset \PP^n$ with
coordinates $x_0,\dots,x_n$ on $\PP^n$. Let $H_{\infty} = \PP^n
\setminus \CC^n$ be the hyperplane at infinity defined by
$x_n=0$. Also consider the closure $\bar{X} \subset \PP^n$ and the
intersection $\bar{X}_{\infty}=\bar{X} \cap H_{\infty}$. We consider
$\bar{X}_{\infty}$ a subvariety of $\PP^{n-1} \cong H_{\infty}$.

\begin{defn}
A smooth affine variety $X \subset \CC^n$ is in \emph{general
  position} if $\bar{X}_{\infty}$ is smooth and both $\bar{X}$ and
$\bar{X}_{\infty}$ are in general position.
\end{defn}

Assume that $X$ is in general position. Let $\nu:\PP^n \setminus \{o\}
\to H_{\infty}$ be the projection from the point $o=(0,\dots,0,1)$. If
$(p,q) \in \mathcal{C}_{\bar{X}}$ then $q \neq o$ since
$\bar{X}_{\infty}$ is smooth. Also, $p \neq \nu(q)$ since
$\bar{X}_{\infty}$ is in general position. Therefore we can define a
map \[g:\mathcal{C}_{\bar{X}} \rightarrow \gr{2}{n+1} : (p,q) \mapsto
\langle p,\nu(q) \rangle,\] mapping a pair $(p,q) \in
\mathcal{C}_{\bar{X}}$ to the line spanned by $p$ and $\nu(q)$. For
the remainder of this section we will use $g$ to denote this map
associated to a variety $X$. In the following lemma we show that for
$x \in X$ the fiber $F_x=\{(x',q) \in \mathcal{C}_{\bar{X}}:x'=x\}$
together with the map $g$ parameterize lines in the Euclidean normal
space $N_xX$ passing through $x$. Recall that for $x \in X \subset
\CC^n$, $(T_xX)_0$ denotes the embedded tangent space translated to
the origin and the Euclidean normal space at $x$ is given by $N_xX=\{z
\in \CC^n:(z-x) \in (T_xX)_0^{\perp}\}$.
\begin{lemma} \label{lemma:lines-in-normalspace}
Let $X \subset \CC^n$ be a smooth variety in general position. As
above we consider $\CC^n \subset \PP^n$ and $X \subset \bar{X} \subset
\PP^n$. Let $x \in X$ and $F_x=\{(x',q) \in
\mathcal{C}_{\bar{X}}:x'=x\}$. Then the map $u \mapsto g(u) \cap
\CC^n$ on $F_x$ defines a one-to-one correspondence between $F_x$ and
the set of lines in $N_xX$ passing through $x$.
\end{lemma}
\begin{proof}
Let $(x,q) \in \mathcal{C}_{\bar{X}}$ with $q=(q_1,\dots,q_{n+1})$ and
$x=(x_1,\dots,x_n,1)$ where $(x_1,\dots,x_n) \in X \subset \CC^n$. The
line $\langle x, \nu(q) \rangle \cap \CC^n$ expressed in coordinates
on $\CC^n$ is given by $\{(x_1,\dots,x_n)+a(q_1,\dots,q_n):a \in
\CC\}$. To show that this line is normal to $X$ at $x$ we need to show
that $(q_1,\dots,q_n) \in (T_xX)_0^{\perp} \subset \CC^n$. Let
$(v_1,\dots,v_n) \in (T_xX)_0$. Then $(x_1+v_1,\dots,x_n+v_n) \in
T_xX$ where $T_xX \subset \CC^n$ is the embedded tangent space of $X$
at $x$. This means that $(x_1+v_1,\dots,x_n+v_n,1) \in
\mathbb{T}_x\bar{X} \subset \PP^n$ and hence
$\sum_{i=1}^{n}(x_i+v_i)q_i+q_{n+1}=0$. Since $(x_1,\dots,x_n,1) \in
\mathbb{T}_x\bar{X}$ we have that $\sum_{i=1}^{n}x_iq_i+q_{n+1}=0$. It
follows that $\sum_{i=1}^nv_iq_i=0$ and we have shown $(q_1\dots,q_n)
\in (T_xX)_0^{\perp}$.

Now let $x \in X \subset \bar{X}$ with $x=(x_1,\dots,x_n,1)$ and
consider a line in $N_xX$ through $x$. In coordinates on $\CC^n$ the
line is given by $\{(x_1,\dots,x_n)+a(q_1,\dots,q_n):a\in \CC\}$ for
some $0 \neq (q_1,\dots,q_n) \in (T_xX)_0^{\perp} \subseteq
\CC^n$. Note that $(q_1,\dots,q_n)$ is unique up to scaling. We need
to show that there is a unique $q_{n+1} \in \CC$ such that $(x,q) \in
\mathcal{C}_{\bar{X}}$ where $q=(q_1,\dots,q_n,q_{n+1})$. Since $x \in
\mathbb{T}_x\bar{X}$, a necessary condition on $q_{n+1} \in \CC$ is
that $\sum_{i=1}^{n}x_iq_i+q_{n+1}=0$. Accordingly we let
$q_{n+1}=-\sum_{i=1}^nx_iq_i$. It remains to show $(x,q) \in
\mathcal{C}_{\bar{X}}$. Since $\{(v_1,\dots,v_n,v_{n+1}) \in
\mathbb{T}_x\bar{X}:v_{n+1}\neq 0\} \subset \mathbb{T}_x\bar{X}$ is a
dense subset, it is enough to show that for all $(v_1,\dots,v_n,1) \in
\mathbb{T}_x\bar{X}$ we have that $\sum_{i=1}^n v_iq_i +
q_{n+1}=0$. Note that $(v_1,\dots,v_n) \in T_xX \subset \CC^n$ and
$(v_1-x_1,\dots,v_n-x_n) \in (T_xX)_0$. Hence
$\sum_{i=1}^n(v_i-x_i)q_i=0$. It follows that $\sum_{i=1}^n v_iq_i +
q_{n+1}=\sum_{i=1}^nx_iq_i+q_{n+1}=0$.
\end{proof}

Consider the projection $\mathfrak{p}: \mathcal{C}_{\bar{X}} \to
\bar{X}$. A \emph{bottleneck} of the affine variety $X$ is a pair of
distinct points $x,y \in X \subset \bar{X}$ such that there exists
$u,v \in \mathcal{C}_{\bar{X}}$ with $\mathfrak{p}(u)=x$,
$\mathfrak{p}(v)=y$ and $g(u)=g(v)$. We will now show that this
definition of bottlenecks is equivalent to the definition given in
\secref{sec:intro} in terms of Euclidean normal spaces.
\begin{lemma}
Let $X \subset \CC^n$ be a smooth variety in general position. A pair
of distinct points $(x,y) \in X \times X$ is a bottleneck if and only
if the line in $\CC^n$ joining $x$ and $y$ is contained in $N_xX \cap
N_yX$.
\end{lemma}
\begin{proof}
If $(x,y) \in X \times X$ is a bottleneck, then there are $u,v \in
\mathcal{C}_{\bar{X}}$ with $g(u)=g(v)$, $\mathfrak{p}(u)=x$ and
$\mathfrak{p}(v)=y$. The line $g(u) \cap \CC^n$ in $\CC^n$ thus
contains $x$ and $y$ and it is contained in $N_xX \cap N_yX$ by
\lemmaref{lemma:lines-in-normalspace}. For the converse, let $x,y \in
X$ be distinct such that the line $l \subset \CC^n$ joining $x$ and
$y$ is contained in $N_xX \cap N_yX$. By
\lemmaref{lemma:lines-in-normalspace} there are $u,v \in
\mathcal{C}_{\bar{X}}$ with $l=g(u)\cap \CC^n$, $l=g(v) \cap \CC^n$,
$\mathfrak{p}(u)=x$ and $\mathfrak{p}(v)=y$. It follows that
$g(u)=g(v)$ and $(x,y)$ is a bottleneck.
\end{proof}

The map $g$ can have double points that do not correspond to actual
bottlenecks of $X$ since we require that $x,y \in X$ lie in the affine
part. Note however that if $u,v \in \mathcal{C}_{\bar{X}}$ with
$g(u)=g(v)$ and $\mathfrak{p}(u) \in \bar{X}_{\infty}$, then
$\mathfrak{p}(v) \in \bar{X}_{\infty}$ as well. Therefore the
extraneous double point pairs of $g$ are in one-to-one correspondence
with double point pairs of the
map \[g_{\infty}:\mathcal{C}_{\bar{X}_{\infty}} \rightarrow \gr{2}{n}
: (p,q) \mapsto \langle p,q \rangle.\] Here
$\mathcal{C}_{\bar{X}_{\infty}}$ is defined with respect to the
embedding $\bar{X}_{\infty} \subset \PP^{n-1}$. This leads us to
consider the double point classes $\mathbb{D}(g)$,
$\mathbb{D}(g_{\infty})$ of $g$ and $g_{\infty}$ and define the
bottleneck degree of $X$ as the difference of the degrees of these
classes.

\begin{defn}
Let $X \subset \CC^n$ be a smooth variety in general position. The
\emph{bottleneck degree} of $X$ is
$\deg{\mathbb{D}(g)}-\deg{\mathbb{D}(g_{\infty})}$ and is denoted by
$\BND{X}$.
\end{defn}

\begin{ex} \label{ex:affine-plane-curve}
Consider a general plane curve $X \subset \CC^2$ of degree $d$ defined
by a polynomial $F \in \CC[x,y]$. Then $\bar{X} \subset \PP^2$ is
defined by the homogenization $\bar{F} \in \CC[x,y,z]$ of $F$ with
respect to $z$. We may assume that $\bar{X}$ is smooth. The map
$g:\bar{X} \to (\PP^2)^*$ is given by $(x,y,z) \mapsto
(-z\bar{F}_y,z\bar{F}_x,x\bar{F}_y-y\bar{F}_x)$. It maps a point $p
\in X$ to the closure of the normal line $N_pX \subset \CC^2$ in
$\PP^2$. The bottlenecks of $X$ are the pairs $(p,q) \in X \times X$
with $p\neq q$ and $N_pX=N_qX$. We shall now consider the other double
point pairs of $g$, that is distinct points $p,q \in \bar{X}$ such
that $g(p)=g(q)$ and $p$ or $q$ lies on the line at infinity
$H_{\infty}$. The latter corresponds to the point $(0,0,1) \in
(\PP^2)^*$. If $p = (x,y,z) \in H_{\infty} \cap \bar{X}$, that is
$z=0$, then $g(p)=(0,0,1)$. Conversely, if $q \in \bar{X}$ and $g(q) =
(0,0,1)$ then $q \in H_{\infty}$ since $q$ is a point on the line
$g(q)$.
The extraneous double points of $g$ are thus exactly the $d$ points
$\bar{X}_{\infty}$, the intersection of $\bar{X}$ with the line at
infinity. This gives rise to $d(d-1)$ extraneous double point pairs at
infinity.
\end{ex}

\begin{prop} \label{prop:affine}
For a smooth affine variety $X \subset \CC^n$ in general
position, \[\BND{X}=\BND{\bar{X}}-\BND{\bar{X}_{\infty}}.\]
\end{prop}
\begin{proof}
By definition $\BND{\bar{X}_{\infty}}=\deg{\mathbb{D}(g_{\infty})}$
and so it remains to prove that
$\BND{\bar{X}}=\deg{\mathbb{D}(g)}$. In other words we need to show
that $\deg{\mathbb{D}(f)}=\deg{\mathbb{D}(g)}$ where $f$ is the
map \[f:\mathcal{C}_{\bar{X}} \rightarrow \gr{2}{n+1} : (p,q) \mapsto
\langle p,q \rangle.\] By the double point formula it is enough to
show that
$f^*f_*[\mathcal{C}_{\bar{X}}]=g^*g_*[\mathcal{C}_{\bar{X}}]$ and
$c(f^*T_{G})=c(g^*T_{G})$ where $G=\gr{2}{n+1}$. Since the Schubert
classes generate $A_*(G)$ as a group the equality
$c(f^*T_{G})=c(g^*T_{G})$ would follow after showing that
$f^*\sigma_{a,b}=g^*\sigma_{a,b}$. We will do this first. As in the
proof of \thmref{thm:misc}, let $bl: \Bl{\Delta}{\PP^n \times
  \PP^n} \to \PP^n \times \PP^n$ be the blow-up of $\PP^n \times
\PP^n$ along the diagonal $\Delta \subset \PP^n \times \PP^n$ and let
$E = bl^{-1}(\Delta)$. Let $\alpha, \beta \in A_*(\PP^n \times
\PP^n)$ be the pullbacks of the hyperplane class of $\PP^n$ under the
two projections and let $\gamma$ be as in \thmref{thm:misc}. By
(\ref{eq:pullback}) \[\gamma^*(\sigma_{a,b}) =\sum_{i=0}^{a-b}
bl^*\alpha^{b+i}bl^*\beta^{a-i} + R,\] where $R=[E]\cdot \delta$ for
some $\delta \in A_*(\Bl{\Delta}{\PP^n \times \PP^n})$. Let
$i:\mathcal{C}_{\bar{X}} \to \Bl{\Delta}{\PP^n \times \PP^n}$ be the
map induced by the inclusion $\mathcal{C}_{\bar{X}} \subset \PP^n
\times \PP^n$ and let $j:\mathcal{C}_{\bar{X}} \to \Bl{\Delta}{\PP^n
  \times \PP^n}$ be induced by the map $\mathcal{C}_{\bar{X}} \to
\PP^n \times \PP^n:(p,q) \mapsto (p,\nu(q))$, where $\nu:\PP^n
\setminus \{o\} \to H_{\infty}$ is the linear projection. Note that
$f=\gamma \circ i$ and $g=\gamma \circ j$. The map $bl\circ i$ is
the identity on $\mathcal{C}_{\bar{X}}$ and $bl \circ
j:\mathcal{C}_{\bar{X}} \to \PP^n \times \PP^n$ is the map $(p,q)
\mapsto (p,\nu(q))$. It follows that $i^*bl^*\alpha = j^*bl^*\alpha$
and $i^*bl^*\beta = j^*bl^*\beta$. Since $\bar{X}$ and
$\bar{X}_{\infty}$ are in general position, $i^*R=j^*R=0$, and we
conclude that $f^*\sigma_{a,b}=g^*\sigma_{a,b}$.

Now write $f_*[\mathcal{C}_{\bar{X}}]=\sum_i e_i\sigma_{n-1-i,i}$ and
$g_*[\mathcal{C}_{\bar{X}}]=\sum_i e_i'\sigma_{n-1-i,i}$ where
$e_i,e_i' \in \ZZ$. Note that $e_i=\deg{f_*[\mathcal{C}_{\bar{X}}]
  \cdot \sigma_{n-1-i,i}}=\deg{[\mathcal{C}_{\bar{X}}] \cdot
  f^*\sigma_{n-1-i,i}}$ and the same way
$e_i'=\deg{[\mathcal{C}_{\bar{X}}] \cdot g^*\sigma_{n-1-i,i}}$. Since
$f^*\sigma_{n-1-i,i}=g^*\sigma_{n-1-i,i}$, we have that $e_i=e_i'$ for
all $i$. It follows that \[f^*f_*[\mathcal{C}_{\bar{X}}]=\sum_i
e_if^*\sigma_{n-1-i,i}=\sum_i
e_i'g^*\sigma_{n-1-i,i}=g^*g_*[\mathcal{C}_{\bar{X}}].\]
\end{proof}

\begin{ex} \label{ex:affine-plane-curve2}
  For a general curve $X \subset \CC^2$ of degree $d$ we
  have \[\BND{X}=d^4-5d^2+4d.\] Namely, the bottleneck degree of
  $\bar{X}$ is given in \cororef{coro:specific-formulas} and putting this
  together with \propref{prop:affine} and
  \exref{ex:affine-plane-curve} we get $\BND{X}=d^4-4d^2+3d -
  d(d-1)=d^4-5d^2+4d$.

 Hence a general line in $\CC^2$ has no bottlenecks, as one might
 expect. For $d=2$ we get that a general conic has 4 bottlenecks,
 which corresponds to 2 pairs of points with coinciding normal
 lines. These lines can be real: Consider the case of a general real
 ellipse and its two principal axes, see \figref{fig:ellipse}.
\begin{figure}[htb]
  \centering
  \includegraphics[trim={0pt 0pt 0pt 0pt},clip,width=150pt]{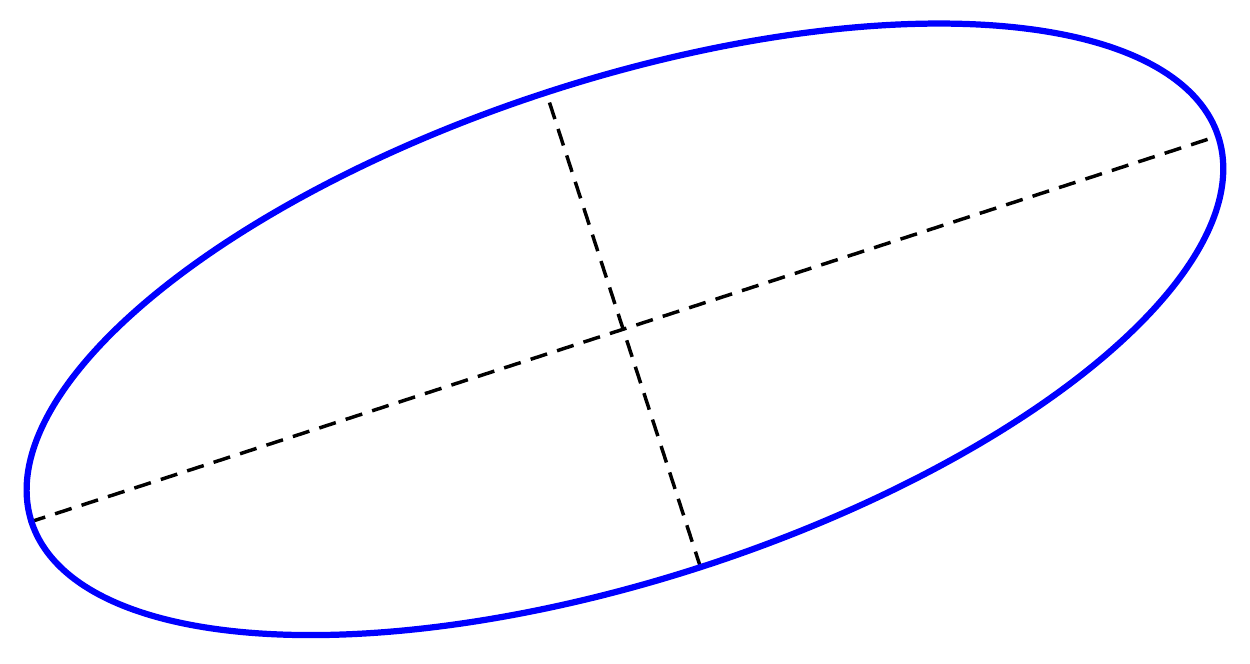}
  \caption{\label{fig:ellipse} Ellipse with principal axes.}
\end{figure}
 
 \end{ex}

\begin{rem}
Let $X$ be a smooth affine variety in general position. As we have
seen, the bottleneck degrees of $\bar{X}$ and $\bar{X}_{\infty}$ are
functions of the polar numbers of these varieties. If
$i:\bar{X}_{\infty} \to \bar{X}$ is the inclusion then the relation
between the polar classes of $\bar{X}$ and those of its hyperplane
section $\bar{X}_{\infty}$ is
$p_j(\bar{X}_{\infty})=i^*p_j(\bar{X})$. This is straightforward to
verify using for example the adjunction formula \cite[Example
3.2.12]{F98}  and the relation between polar classes and Chern classes
(\ref{eq:polar-chern}). This means that the polar numbers of $\bar{X}$
and $\bar{X}_{\infty}$ are the same in the sense that
$\deg{p_1(\bar{X}_{\infty})^{a_1} \dots
  p_{m-1}(\bar{X}_\infty)^{a_{m-1}}}=\deg{p_1(\bar{X})^{a_1} \dots
  p_{m-1}(\bar{X})^{a_{m-1}}}$ for any $a_1,\dots,a_{m-1} \in \NN$
such that $\sum_{j=1}^{m-1} j\cdot a_j \leq m-1$. As a consequence,
\propref{prop:affine} may be used  to express the bottleneck degree
of $X \subset \CC^n$ in terms of the polar numbers of its closure
$\bar{X} \subset \PP^n$.
\end{rem}

\begin{rem} \label{rem:homotopy}
Let $g_1,\dots,g_k \in \CC[x_1,\dots,x_n]$ be a system of polynomials
of degrees $d_1,\dots,d_k$ which define a complete intersection $X
\subset \CC^n$. Suppose that the bottlenecks of $X$ are known. If $X$
is general enough to have the maximal number of bottlenecks, we may
compute the isolated bottlenecks of any other complete intersection $Y
\subset \CC^n$ defined by polynomials $f_1,\dots,f_k \in
\CC[x_1,\dots,x_n]$ of the same degrees $d_1,\dots,d_k$. We propose to
do this by a parameter homotopy from $X$ to $Y$. For background on
homotopy methods see for example \cite{SW05}. Suppose that both $X$
and $Y$ are smooth. Let $h_i(x)=(1-t)f_i(x)+\gamma tg_i(x)$ where
$\gamma \in \CC$ is random and $t$ is the homotopy parameter. The
homotopy paths are tracked from the bottlenecks of $X$ at $t=1$ to the
bottlenecks at $Y$ at $t=0$. Introduce new variables $y_1,\dots,y_n$
and $\lambda_1,\dots,\lambda_k,\mu_1,\dots,\mu_k$. The parameter
homotopy is then the following square system of equations in $2(n+k)$
variables:
\[
\begin{array}{l}
  h_1(x)=\dots=h_k(x)=0, \\
  h_1(y)=\dots=h_k(y)=0, \\
  y-x=\sum_{i=1}^k \lambda_i \nabla h_i(x), \\
  y-x=\sum_{i=1}^k \mu_i \nabla h_i(y). \\
\end{array}
\]
For the starting points of the homotopy we need the bottleneck pairs
$(x,y)$ of $X$. To find the $\lambda_1,\dots,\lambda_k$ and
$\mu_1,\dots,\mu_k$ corresponding to a bottleneck pair $(x,y)$ one
would need to solve the linear systems $y-x=\sum_{i=1}^k \lambda_i
\nabla g_i(x)$ and $y-x=\sum_{i=1}^k \mu_i \nabla g_i(y)=0$.

Along similar lines, \cite{E18} presents an efficient homotopy to
compute bottlenecks of affine varieties.
\end{rem}

\section{Examples} \label{sec:examples}

\begin{ex}
Consider the space curve in $\mathbb{C}^3$ given by the intersection
of these two hypersurfaces:
\[   x^3-3xy^2-z = 0\]
\[ x^2 + y^2 + 3z^2 - 1=0 .\]

\begin{figure}[htb]
  \centering
  \includegraphics[trim={0pt 50pt 0pt 60pt},clip,width=150pt]{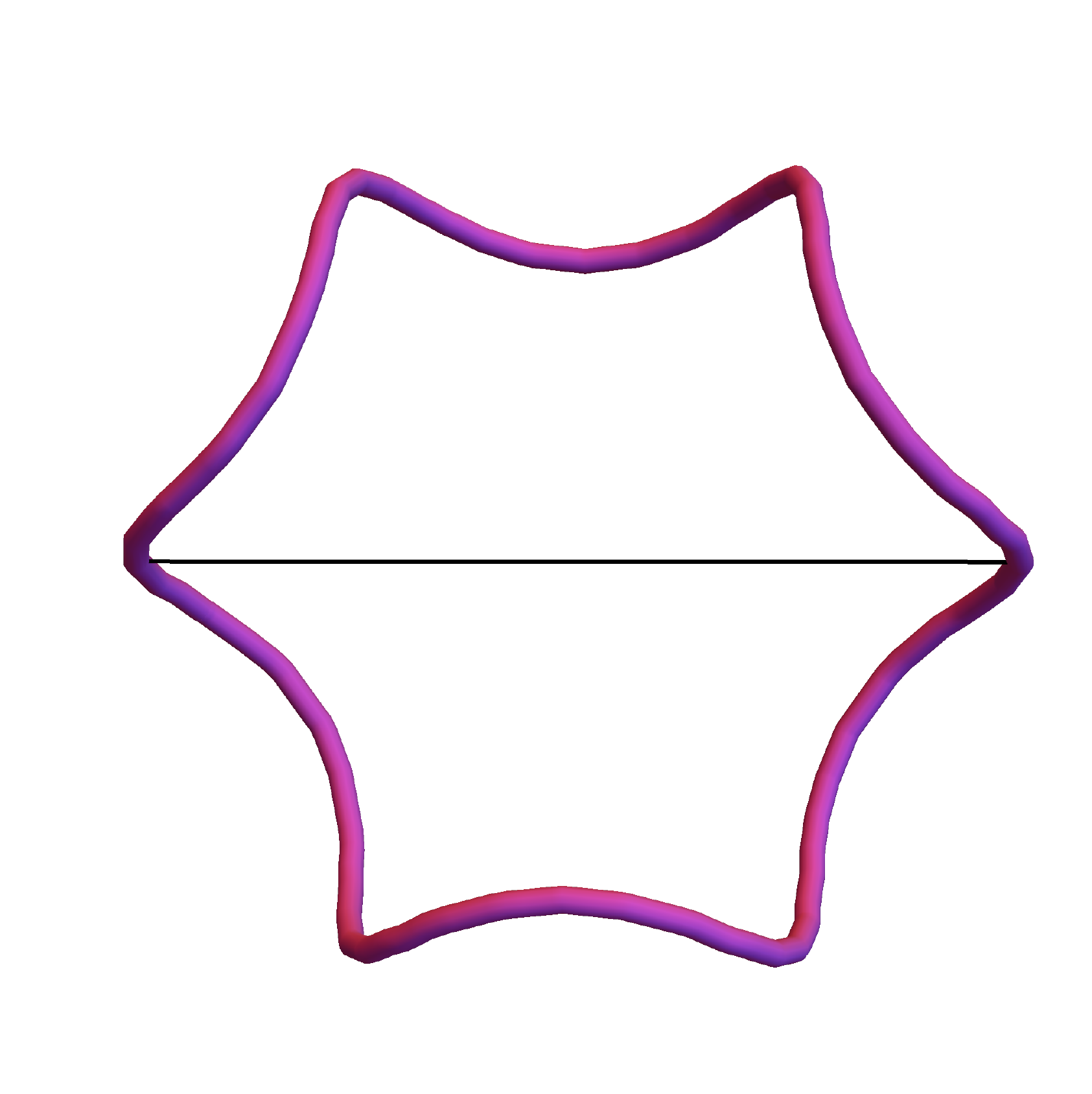}
  \caption{\label{fig:spacecurve} A space curve of degree 6 shown with
    one of its bottleneck lines joining the points $(0,-1,0)$ and $(0,1,0)$.}
\end{figure}

As computed in Macaulay2, the ideal of the bottleneck variety (with
the diagonal removed) associated to this affine curve has dimension 0
and degree 480. The curve is the complete intersection of two surfaces
of degrees $d_1=2$ and $d_2=3$.

Now consider a smooth complete intersection curve $X \subset \CC^3$
cut out by surfaces of degree $d_1$ and $d_2$. Assume that $X$ is in
general position. By  \cororef{coro:specific-formulas}  the
bottleneck degree of $\bar{X}$ is given by
$\epsilon_0^2+d^2-5\deg{p_1}-2d$, where $\epsilon_0=d+\deg{p_1}$ and
$d=d_1d_2$. Using for example the adjunction formula, \cite[Example 3.2.12]{F98}, one can see that $c_1(T_X)=(4-d_1-d_2)h$, where $h \in
A_0(X)$ is the hyperplane class. Also, by (\ref{eq:polar-chern}) we
have $p_1 = 2h-c_1(T_X)=(d_1+d_2-2)h$. Thus
$\deg{p_1}=(d_1+d_2-2)d_1d_2$. By Proposition \ref{prop:affine}, to
obtain the bottleneck degree of the affine variety $X$ we subtract
$\BND{\bar{X}_{\infty}}$ from $\BND{\bar{X}}$. In this case, we have
\[ \BND{\bar{X}_{\infty}}=d_1d_2(d_1d_2-1). \] 
We obtain the following formula for the bottleneck degree of a smooth
complete intersection curve $X \subset \CC^3$ in general position:
\[ \BND{X}=d_1^4d_2^2+2d_1^3d_2^3+d_1^2d_2^4-2d_1^3d_2^2-2d_1^2d_2^3+d_1^2d_2^2-5d_1^2d_2-5d_1d_2^2+9d_1d_2.  \] 
Substituting $d_1=2$ and $d_2=3$, we obtain $\BND{X}=480$, in
agreement with the Macaulay2 computation for the sextic curve above.
\end{ex}

\begin{ex} \label{ex:surfaces}
Let $X \subset \CC^3$ be a general surface of degree
$d$. Then \[\BND{X} = d^6-2d^5+3d^4-15d^3+26d^2-13d.\] To see this use
\propref{prop:affine}. Apply \cororef{coro:specific-formulas} to get
$\BND{\bar{X}}$ and the bottleneck degree of the planar curve
$\bar{X}_{\infty}$. 
\end{ex}

\begin{ex}\label{ex:qsurface}
Consider the quartic surface $X \subset \mathbb{C}^3$ defined by the equation 
\[ (0.3x^2 + 0.5z + 0.3x + 1.2y^2 - 1.1)^2 +
           (0.7(y - 0.5x)^2 + y + 1.2z^2 - 1)^2 = 0.3.\] For a general
quartic surface in $\mathbb{C}^3$, the bottleneck degree is 2220 by
\exref{ex:surfaces}. In this case, $\BND{X}=1390$ was found using the
Julia package \emph{HomotopyContinuation.jl}
\cite{BreidingTimme}. Among the $1390$ solutions are $49$ distinct
real bottleneck pairs. The quartic with its bottlenecks is shown in
\figref{fig:quarticsurf}.

\end{ex}

\begin{figure}[htb]
  \centering
  \includegraphics[trim={0pt 250pt 0pt 250pt},clip,width=150pt]{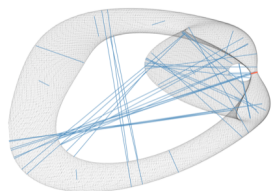}
  \caption{\label{fig:quarticsurf} The quartic surface of Example \ref{ex:qsurface} shown with
    its real bottleneck lines. The shortest bottleneck line is shown
    in red. The figure was produced by Sascha Timme using the Julia
    package \emph{HomotopyContinuation.jl} \cite{BreidingTimme}.}
\end{figure}

\begin{ex}
Consider the ellipsoid $X \subset \mathbb{C}^3$ defined by the equation 
\[ 36x^2+9y^2+4z^2=36.   \]
For a general quadric surface in $\mathbb{C}^3$, the bottleneck degree
is 6 by \exref{ex:surfaces}.  In this case, there are indeed three
bottleneck pairs, all with real coordinates.  The pairs occur on each
of the coordinate axes, at $\{ (-1,0,0), (1,0,0) \}$, $\{ (0,-2,0),
(0,2,0) \}$ and $\{ (0,0,-3),(0,0,3) \}$.

If we set two axes to be the same length, as in the equation of the spheroid
\[4x^2+y^2+z^2=4,\]
then there is only one isolated bottleneck pair: $\{ (-1,0,0), (1,0,0)
\}$. The rest of the bottlenecks are part of an infinite
locus. Intersecting with the plane $\{ x=0 \}$ which is normal to the
spheroid, we obtain the circle $\{ y^2+z^2=4 \}$ and every antipodal
pair of points of the circle is a bottleneck.
\end{ex}

\bibliographystyle{plain} \bibliography{collection}

\begin{thebibliography}{10}

\bibitem{ACKMRW17}
E.~Aamari, F.~Chazal, J.~Kim, B.~Michel, A.~Rinaldo, and L.~Wasserman.
\newblock Estimating the reach of a manifold.
\newblock {\em Electron. J. Statist.}, 13(1):1359--1399, 2019.

\bibitem{Aamari2018}
Eddie Aamari and Cl{\'e}ment Levrard.
\newblock Stability and minimax optimality of tangential {D}elaunay complexes
  for manifold reconstruction.
\newblock {\em Discrete {\&} Computational Geometry}, 59(4):923--971, Jun 2018.

\bibitem{ACM05}
L.~Alberti, G.~Comte, and B.~Mourrain.
\newblock Meshing implicit algebraic surfaces: the smooth case.
\newblock {\em Mathematical Methods for Curves and Surfaces: Troms{\o} 2004},
  2005.

\bibitem{minimaxvolume}
Ery Arias-Castro, Beatriz Pateiro-López, and Alberto Rodríguez-Casal.
\newblock Minimax estimation of the volume of a set under the rolling ball
  condition.
\newblock {\em Journal of the American Statistical Association},
  114(527):1162--1173, 2019.

\bibitem{balakrishnan2012}
Sivaraman Balakrishnan, Alesandro Rinaldo, Don Sheehy, Aarti Singh, and Larry
  Wasserman.
\newblock Minimax rates for homology inference.
\newblock In Neil~D. Lawrence and Mark Girolami, editors, {\em Proceedings of
  the Fifteenth International Conference on Artificial Intelligence and
  Statistics}, volume~22 of {\em Proceedings of Machine Learning Research},
  pages 64--72, La Palma, Canary Islands, 21--23 Apr 2012. PMLR.

\bibitem{Bank2001}
B.~Bank, M.~Giusti, J.~Heintz, and G.M. Mbakop.
\newblock Polar varieties and efficient real elimination.
\newblock {\em Mathematische Zeitschrift}, 238(1):115--144, Sep 2001.

\bibitem{BW09}
R.~G. Baraniuk and M.~B. Wakin.
\newblock Random projections of smooth manifolds.
\newblock {\em Foundations of Computational Mathematics}, 9:51--77, 2009.

\bibitem{Bates2014}
Daniel~J. Bates, Wolfram Decker, Jonathan~D. Hauenstein, Chris Peterson,
  Gerhard Pfister, Frank-Olaf Schreyer, Andrew~J. Sommese, and Charles~W.
  Wampler.
\newblock Comparison of probabilistic algorithms for analyzing the components
  of an affine algebraic variety.
\newblock {\em Appl. Math. Comput.}, 231(C):619--633, March 2014.

\bibitem{BATES2013493}
Daniel~J. Bates, David Eklund, and Chris Peterson.
\newblock Computing intersection numbers of {C}hern classes.
\newblock {\em Journal of Symbolic Computation}, 50:493 -- 507, 2013.

\bibitem{bertini}
D.J. Bates, J.D. Hauenstein, A.J. Sommese, and C.W. Wampler.
\newblock Bertini: Software for numerical algebraic geometry.
\newblock \url{dx.doi.org/10.7274/R0H41PB5}.

\bibitem{B06}
Avrim Blum.
\newblock Random projection, margins, kernels, and feature-selection.
\newblock In Craig Saunders, Marko Grobelnik, Steve Gunn, and John
  Shawe-Taylor, editors, {\em Subspace, Latent Structure and Feature
  Selection}, pages 52--68, Berlin, Heidelberg, 2006. Springer Berlin
  Heidelberg.

\bibitem{BG11}
Jean-Daniel Boissonnat and Arijit Ghosh.
\newblock Triangulating smooth submanifolds with light scaffolding.
\newblock {\em Mathematics in Computer Science}, 4(4):431, Sep 2011.

\bibitem{Boissonnat2014}
Jean-Daniel Boissonnat and Arijit Ghosh.
\newblock Manifold reconstruction using tangential {D}elaunay complexes.
\newblock {\em Discrete {\&} Computational Geometry}, 51(1):221--267, Jan 2014.

\bibitem{BO05}
Jean-Daniel Boissonnat and Steve Oudot.
\newblock Provably good sampling and meshing of surfaces.
\newblock {\em Graphical Models}, 67(5):405 -- 451, 2005.
\newblock Solid Modeling and Applications.

\bibitem{BKSW18}
P.~Breiding, S.~Kali\v{s}nik, B.~Sturmfels, and M.~Weinstein.
\newblock Learning algebraic varieties from samples.
\newblock {\em Revista Matem{\'a}tica Complutense}, 31(3):545--593, 2018.

\bibitem{reach-curve2019}
Paul Breiding and Sascha Timme.
\newblock The reach of a plane curve.
\newblock \url{ https://www.JuliaHomotopyContinuation.org/examples/reach-curve/
  }.
\newblock Accessed: October 24, 2019.

\bibitem{BreidingTimme}
Paul Breiding and Sascha Timme.
\newblock Homotopycontinuation.jl - a package for solving systems of polynomial
  equations in julia.
\newblock {\em CoRR}, abs/1711.10911, 2017.

\bibitem{Burgisser2007}
Peter B{\"u}rgisser and Martin Lotz.
\newblock The complexity of computing the {H}ilbert polynomial of smooth
  equidimensonal complex projective varieties.
\newblock {\em Foundations of Computational Mathematics}, 7(1):59--86, Feb
  2007.

\bibitem{C08}
Kenneth~L. Clarkson.
\newblock Tighter bounds for random projections of manifolds.
\newblock In {\em Proceedings of the Twenty-fourth Annual Symposium on
  Computational Geometry}, SCG '08, pages 39--48, New York, NY, USA, 2008. ACM.

\bibitem{Cuevas2007ANA}
Antonio Cuevas, Ricardo Fraiman, and Alberto Rodr\'iguez-Casal.
\newblock A nonparametric approach to the estimation of lengths and surface
  areas.
\newblock 2007.

\bibitem{DG03}
Sanjoy Dasgupta and Anupam Gupta.
\newblock An elementary proof of a theorem of {Johnson} and {Lindenstrauss}.
\newblock {\em Random Structures \& Algorithms}, 22(1):60--65, 2003.

\bibitem{bnscript}
S.~Di~Rocco, D.~Eklund, and M.~Weinstein.
\newblock A {Macaulay2} script to compute bottleneck degree.
\newblock Available at \url{https://bitbucket.org/daek/bottleneck_script}.

\bibitem{DEP2017}
Sandra Di~Rocco, David Eklund, and Chris Peterson.
\newblock Numerical polar calculus and cohomology of line bundles.
\newblock {\em Adv. Appl. Math.}, 100:148--162, 2017.

\bibitem{DiRocco2011}
Sandra Di~Rocco, David Eklund, Chris Peterson, and Andrew~J. Sommese.
\newblock Chern numbers of smooth varieties via homotopy continuation and
  intersection theory.
\newblock {\em J. Symb. Comput.}, 46(1):23--33, January 2011.

\bibitem{DHOS16}
J.~Draisma, E.~Horobe\cb{t}, G.~Ottaviani, B.~Sturmfels, and R.~Thomas.
\newblock The {Euclidean} distance degree of an algebraic variety.
\newblock {\em Foundations of Computational Mathematics}, 16(1), 2016.

\bibitem{DEHH18}
E.~Dufresne, P.B. Edwards, H.A. Harrington, and J.D. Hauenstein.
\newblock Sampling real algebraic varieties for topological data analysis.
\newblock {\em arXiv:1802.07716}, 2018.

\bibitem{EH16}
D.~Eisenbud and J.~Harris.
\newblock {\em 3264 and All That: A Second Course in Algebraic Geometry}.
\newblock Cambridge University Press, first edition, 2016.

\bibitem{E18}
D.~Eklund.
\newblock The numerical algebraic geometry of bottlenecks.
\newblock {\em arXiv:1804.01015}, 2018.

\bibitem{EJP}
David Eklund, Christine Jost, and Chris Peterson.
\newblock A method to compute {S}egre classes of subschemes of projective
  space.
\newblock {\em Journal of Algebra and Its Applications}, 12(02), 2013.

\bibitem{F59}
H.~Federer.
\newblock Curvature measures.
\newblock {\em Transactions of the American Mathematical Society},
  93(2):418--491, 1959.

\bibitem{F98}
W.~Fulton.
\newblock {\em Intersection theory}.
\newblock Springer, 2nd edition, 1998.

\bibitem{Genovese2012}
Christopher~R. Genovese, Marco Perone-Pacifico, Isabella Verdinelli, and Larry
  Wasserman.
\newblock Minimax manifold estimation.
\newblock {\em J. Mach. Learn. Res.}, 13:1263--1291, May 2012.

\bibitem{M2}
Daniel~R. Grayson and Michael~E. Stillman.
\newblock Macaulay2, a software system for research in algebraic geometry.
\newblock Available at \url{http://www.math.uiuc.edu/Macaulay2/}.

\bibitem{Schubert2Source}
Daniel~R. Grayson, Michael~E. Stillman, Stein~A. Str{\o}mme, David Eisenbud,
  and Charley Crissman.
\newblock {Schubert2: computations of characteristic classes for varieties
  without equations. Version~0.7}.
\newblock A \emph{Macaulay2} package available at
  \url{https://github.com/Macaulay2/M2/tree/master/M2/Macaulay2/packages}.

\bibitem{H92}
J.~Harris.
\newblock {\em Algebraic Geometry}.
\newblock Springer, first edition, 1992.

\bibitem{BHW08}
Chinmay Hegde, Michael Wakin, and Richard Baraniuk.
\newblock Random projections for manifold learning.
\newblock In J.~C. Platt, D.~Koller, Y.~Singer, and S.~T. Roweis, editors, {\em
  Advances in Neural Information Processing Systems 20}, pages 641--648. Curran
  Associates, Inc., 2008.

\bibitem{HW18}
Emil Horobe\cb{t} and Madeleine Weinstein.
\newblock Offset hypersurfaces and persistent homology of algebraic varieties.
\newblock {\em Computer Aided Geometric Design}, 74, 2018.

\bibitem{WolframMathematica}
Wolfram~Research{,} Inc.
\newblock Mathematica, {V}ersion 12.0.
\newblock Champaign, IL, 2019.

\bibitem{J78}
Kent~W. Johnson.
\newblock Immersion and embedding of projective varieties.
\newblock {\em Acta Math.}, 140:49--74, 1978.

\bibitem{Kim2015TightMR}
Arlene K.~H. Kim and Harrison~H. Zhou.
\newblock Tight minimax rates for manifold estimation under {H}ausdorff loss.
\newblock 2015.

\bibitem{KUKELOVA2010}
Zuzana Kukelova, Martin Byröd, Klas Josephson, Tomas Pajdla, and Kalle
  Åström.
\newblock Fast and robust numerical solutions to minimal problems for cameras
  with radial distortion.
\newblock {\em Computer Vision and Image Understanding}, 114(2):234 -- 244,
  2010.
\newblock Special issue on Omnidirectional Vision, Camera Networks and
  Non-conventional Cameras.

\bibitem{L78}
Dan Laksov.
\newblock Residual intersections and {Todd's} formula for the double locus of a
  morphism.
\newblock {\em Acta Math.}, 140:75--92, 1978.

\bibitem{lasserre_2015}
Jean~Bernard Lasserre.
\newblock {\em An Introduction to Polynomial and Semi-Algebraic Optimization}.
\newblock Cambridge Texts in Applied Mathematics. Cambridge University Press,
  2015.

\bibitem{MCKELVEY1997411}
Richard~D McKelvey and Andrew McLennan.
\newblock The maximal number of regular totally mixed {N}ash equilibria.
\newblock {\em Journal of Economic Theory}, 72(2):411 -- 425, 1997.

\bibitem{Mehta2018TheLS}
Dhagash Mehta, Tianran Chen, Tingting Tang, and Jonathan~D. Hauenstein.
\newblock The loss surface of deep linear networks viewed through the algebraic
  geometry lens.
\newblock {\em ArXiv}, abs/1810.07716, 2018.

\bibitem{Minimair04}
Manfred Minimair and Michael~P. Barnett.
\newblock Solving polynomial equations for chemical problems using {G}röbner
  bases.
\newblock {\em Molecular Physics}, 102(23-24):2521--2535, 2004.

\bibitem{NSW08}
P.~Niyogi, S.~Smale, and S.~Weinberger.
\newblock Finding the homology of submanifolds with high confidence from random
  samples.
\newblock {\em Discrete Comput. Geom.}, 39(1-3):419--441, 2008.

\bibitem{P78}
R.~Piene.
\newblock Polar classes of singular varieties.
\newblock {\em Annales Scientifiques de l'\'Ecole Normale Sup\'erieure},
  4(11):247--276, 1978.

\bibitem{P15}
R.~Piene.
\newblock Polar varieties revisited.
\newblock {\em Lecture Notes in Comput. Sci 8942: Computer algebra and
  polynomials}, pages 139--150, 2015.

\bibitem{SafeyElDin}
Mohab Safey El~Din and Pierre-Jean Spaenlehauer.
\newblock Critical point computations on smooth varieties: Degree and
  complexity bounds.
\newblock In {\em Proceedings of the ACM on International Symposium on Symbolic
  and Algebraic Computation}, ISSAC '16, pages 183--190, New York, NY, USA,
  2016. ACM.

\bibitem{SW05}
A.~J. Sommese and C.~W. Wampler.
\newblock {\em The Numerical Solution of Systems of Polynomials Arising in
  Engineering and Science}.
\newblock World Scientific, 2005.

\bibitem{Sturmfels02}
Bernd Sturmfels.
\newblock Solving systems of polynomial equations.
\newblock In {\em AMERICAN MATHEMATICAL SOCIETY, CBMS REGIONAL CONFERENCES
  SERIES, NO 97}, 2002.

\bibitem{V11}
N.~Verma.
\newblock A note on random projections for preserving paths on a manifold.
\newblock {\em UC San Diego, Tech. Report CS2011-0971}, 2011.

\bibitem{wampler_sommese_2011}
Charles~W. Wampler and Andrew~J. Sommese.
\newblock Numerical algebraic geometry and algebraic kinematics.
\newblock {\em Acta Numerica}, 20:469–567, 2011.

\end{thebibliography}

\end{document}